\documentclass[reqno,a4paper, 11pt]{amsart}

\usepackage[a4paper=true,pdfpagelabels]{hyperref}
\usepackage{graphicx}
\usepackage{latexsym}
\usepackage{mathabx}
\usepackage{amsmath}
\usepackage{amssymb}
\usepackage[T1]{fontenc}
\usepackage[utf8]{inputenc}
\usepackage{lmodern}

\usepackage{color}   

\newtheorem{theorem}{Theorem}[section]
\newtheorem{lemma}[theorem]{Lemma}

\newtheorem{proposition}[theorem]{Proposition}

\newtheorem{lettertheorem}{Theorem}
\newtheorem{letterlemma}[lettertheorem]{Lemma}
\newtheorem{letterproposition}[lettertheorem]{Proposition}
\theoremstyle{definition}

\theoremstyle{remark}

\numberwithin{equation}{section}

\newcommand{\set}[1]{\left\{#1\right\}}
\newcommand{\abs}[1]{\left | #1\right |}
\newcommand{\nm}[1]{\left \| #1 \right \|}

\newcommand{\B}{\mathcal{B}}
\newcommand{\D}{\mathbb{D}}
\newcommand{\DD}{\widehat{\mathcal{D}}}
\newcommand{\Dd}{\widecheck{\mathcal{D}}}

\newcommand{\DDD}{\mathcal{D}}

\newcommand{\N}{\mathbb{N}}

\newcommand{\C}{\mathbb{C}}

\renewcommand{\phi}{\varphi}

\def\BMOA{\mathord{\rm BMOA}}

\def\a{\alpha}       \def\b{\beta}        \def\g{\gamma}
           
     \def\om{\omega}      
       \def\t{\theta}       
                  \def\z{\zeta}

\def\omg{\widehat{\omega}}

\renewcommand{\H}{\mathcal{H}}

\newenvironment{Prf}{\noindent{\emph{Proof of}}}
{\hfill$\Box$ }

\newcommand{\Apq}[2]{A^{#1,#2}_\om}

\begin{document}
	\title[Duality of  mixed norm spaces induced by radial  one-sided doubling weight]{Duality of  mixed norm spaces induced by on radial  one-sided  doubling weight}

	\author[A. moreno]{\'Alvaro Miguel Moreno}
	\address{Departamento de Analisis Matem\'atico, Universidad de M\'alaga, Campus de Teatinos, 
		29071 Malaga, Spain}
	\email{alvarommorenolopez@uma.es}
	
	\author[J. A. Pel\'aez]{Jos\'e \'Angel Pel\'aez}
	\address{Departamento de Analisis Matem\'atico, Universidad de M\'alaga, Campus de Teatinos, 
		29071 Malaga, Spain}
	\email{japelaez@uma.es}
	
	\thanks{This research is supported in part by Ministerio de Ciencia e Innovaci\'on, Spain, project PID2022-136619NB-I00; La Junta de Andaluc{\'i}a, project FQM210.}
	
	\subjclass{32C37, 46E30, 30H30, 30H35, 30H99, 47G10}
	\keywords{ Dual space, Mixed norm space, Radial doubling weight, Bergman Projection, Bloch space, BMOA space}
	
	\date{\today}
	\begin{abstract}
		For $0<p,q<\infty$ and $\omega$ a radial weight, the space $L^{p,q}_\omega$ consists of complex-valued measurable functions $f$ on the unit disk such that
		$$
		\| f\|_{L^{p,q}_\omega}^q = \int_0^1 \left (\frac{1}{2\pi}\int_0^{2\pi}|f(re^{i\theta})|^pd\theta \right )^{\frac{q}{p}}r\omega(r)\,dr, 
		$$
		and the mixed norm space $A^{p,q}_\omega$ is the subset of $L^{p,q}_\omega$ consisting of analytic functions. \\
		 We say that a radial weight $\omega$ belongs to $\widehat{\mathcal{D}}$ if there exists $C=C(\omega)>0$ such that 
$$\int_r^1\omega(s)ds \leq C \int_{\frac{1+r}{2}}^1\omega(s)\,ds \,\, \text{for every}\,\, 0\leq r <1.$$ 

We describe the dual space of $A^{p,q}_\omega$ for every $0<p,q<\infty$ and  $\omega\in\widehat{\mathcal{D}}$. Later on, 
		we  apply the obtained description of the dual space of $A^{p,q}_\omega$ to prove that the Bergman projection induced by $\omega$, $P_\omega$, is bounded on $L^{p,q}_\omega$ for $1<p,q<\infty$ and 
 $\omega\in \widehat{\mathcal{D}}$. Besides, we also prove that $P_\omega$ and the corresponding maximal Bergman projection $P_\omega^+$ are not simultaneously bounded on $L^{p,q}_\om$ for $1<p,q<\infty$ and 
 $\omega\in \widehat{\mathcal{D}}$.
	\end{abstract}
	\maketitle
	\section{Introduction and basic notation}
	Let $\H(\D)$ denote the space of analytic functions in the unit disc $\D=\{z\in\C:|z|<1\}$ of the complex plane. For a non-negative function $\om\in L^1([0,1))$, 
the extension to $\D$, defined by $\om(z)=\om(|z|)$ for all $z\in\D$, 
is called a radial weight. For $0<p\le\infty, 0<q<\infty$ and such an $\omega$,  the weighted mixed norm Lebesgue  space $L^{p,q}_\om$ consist of measurable functions $f$ on $\D$ such that
	$$
	\nm{f}_{L^{p,q}_\om}^q =\int_0^1 M_p^q(r,f) r\om(r)\,dr,
	$$
	where 
	$$	M_p(r,f) =\left (\frac{1}{2\pi}\int_0^{2\pi}\abs{f(re^{i\t})}^p d\t \right )^{\frac{1}{p}},\quad 0<p<\infty,\,\, 0\leq r < 1,$$
and
$$ M_\infty(r,f) = \sup\limits_{\abs{z}=r}\abs{f(z)}, \quad 0\leq r < 1.$$
The space
	$L^{p,\infty}_\om$ consists of measurable functions $f$ on $\D$ such that
	$$
	\nm{f}_{L^{p,\infty}_\om} = \sup\limits_{0\leq r < 1} M_p(r,f)\omg(r),
	$$
where  $\widehat{\om}(z)=\int_{|z|}^1\om(s)\,ds$. The corresponding weighted 
  mixed norm space is $\Apq{p}{q} = L^{p,q}_\om \cap \H(\D)$.  
	If  $0<p=q<\infty$, we denote    $L^{p,p}_\om = L^p_\om$ and  $\Apq{p}{p} = A^p_\om$, which is nothing but the Bergman space induced by radial weight $\om$.  Throughout this paper we assume that $\widehat{\om}(z)> 0$ for all $z\in\D$, for otherwise $A^{p,q}_\om=\H(\D)$. \\
	\par
	A radial weight $\om$ belongs to $\DD$ if there exists $C=C(\om)>0$ such that 
	$$
	\omg(r)\le C\omg\left(\frac{1+r}{2}\right),\quad 0\le r<1.
	$$
	We say that a radial weight $\om\in\Dd$ if there exists $K=K(\om)>1$ and $C=C(\om)>1$ such that 
	\begin{equation}\label{eq: def Dcheck}
	\omg(r)\ge C\omg\left(1-\frac{1-r}{K}\right),\quad 0\le r<1.
	\end{equation}
	We denote $\mathcal{D}=\DD\cap\Dd$ for short, and we simply say that $\omega$ is a doubling weight if $\om\in\DDD$. 
It includes the standard radial weights  $\nu_\a(z) =(\alpha+1) \left (1-\abs{z}^2 \right )^{\a}$, $\a>-1$,  while exponentially decaying weights do not belong to~$\DD$. Membership in $\DD$ or $\Dd$ does not require differentiability, continuity, or strict positivity. In fact, weights in these classes may vanish on a relatively large portion of each  annulus $\{z : r \le |z| < 1\}$ of $\D$. The classes $\DD$ and $\DDD$ are known to arise naturally in various contexts of the concrete operator theory
on spaces of analytic functions. For instance, they are closely connected with Littlewood–Paley estimates, bounded Bergman projections, and the Dostanić's problem~\cite{advances}. For basic properties and illustrative examples of weights in these and other related classes, see \cite{PR, PelSum14, PR2020, advances} and the references therein.

Mixed norm spaces induced by standard weights first appeared in Hardy and Littlewood’s paper on properties of the integral means~\cite{HLMZ1932}, although the spaces themselves were not explicitly defined until Flett’s works~\cite{F1,F2}. Since then, they have been extensively studied by many authors~\cite{ACRPComp19, AhJet, BlascoCan95, BKM99, Gad,PelRatSie}. For example, such spaces arise naturally in the study of coefficient multipliers on Hardy and weighted spaces, as well as in the analysis of the generalized Hilbert operator on weighted Bergman spaces~\cite{GaGiPeSis, PelRathg, PelSeco}.

 Given a quasi-norm vector space $X\subset \H(\D)$ we denote its dual space (also called topological dual space) by $X^\ast$, that is, the space of continuous linear functionals $L: \, X\to \mathbb{C}$. It is worth  mentioning that it is a Banach space with the norm $\| L\|=\sup\{ |L(f)|: \nm{f}_{X}\le 1\}$. 
If $H\subset \H(\D)$ is a Hilbert space and $Y\subset \H(\D)$ is a normed vector space, we write 
$X^\ast    \simeq Y$ under the $H$-pairing with equivalence of norms,  if any $L\in X^\ast$ satisfies that $L(f)=\langle f,g\rangle_{H}$ for some $g\in Y$ and there are  constants $C_1>0$ and $C_2>0$ such that
$ C_1 \| L\|\le \|g\|_{Y}\le C_2  \| L\|$.\\
For instance, in the case of the classical Hardy spaces $H^p$ it is known the relation
 $(H^p)^\ast \simeq H^{p'}$ $\left( \frac{1}{p}+ \frac{1}{p'}\right)$ if $1<p<\infty$ \cite{zhu} and  the  Fefferman's duality relation $\left (H^1\right )^\ast\simeq\BMOA$ \cite{girela},  under the $H^2$-pairing
$$
\langle f,g \rangle_{H^2}=\lim_{r\to 1^-} \sum_{k=0}^\infty \widehat{f}(k)\overline{\widehat{g}(k)}r^k, \quad f(z)= \sum_{k=0}^\infty \widehat{f}(k)z^k, \,  g(z)= \sum_{k=0}^\infty \widehat{g}(k)z^k \in \H(\D).
$$
Here and on the following, $\BMOA$ is the space consisting of $H^1$ functions with bounded mean oscillation on the boundary.

Within the framework of weighted Bergman spaces, descriptions of $(A^p_\omega)^\ast$ play a crucial role in studying the boundedness and compactness of several significant operators, including the Bergman projection~\cite{advances}, Toeplitz operators~\cite{PelRatSie}, and generalized Hilbert operators~\cite{GaGiPeSis, PelRathg}. They also find applications in function spaces, for example in the case of standard weighted mixed norm spaces, they can be used to characterize the space of coefficient multipliers~\cite{AhJet}.

The main objective of this work is to obtain useful descriptions of the dual space $\left ( \Apq{p}{q}\right )^\ast$, for any $0<p,q<\infty$ and $\om\in\DD$.  In order to state the main result of the paper some notation is needed. For any radial weight $\omega$,  
the  $A^2_\om$ pairing is defined as
$$\langle f,g \rangle_{A^2_\omega}=\lim_{r\to 1^-} \sum_{k=0}^\infty \widehat{f}(k)\overline{\widehat{g}(k)} \om_{2k+1}r^k, \quad f,g \in \H(\D),$$
where $\omega_x = \int_0^1 r^x \,\omega(r)\,dr$, $x \geq 0$,  are the moments induced by $\omega$.

For $0<q\le\infty$, a radial weight $\om$ and $X\subset \H(\D)$ a quasi-normed space, we denote
	$$
	X(q,\om)=\left\{f\in\H(\D):\|f\|^{q}_{X(q,\om)}=\int_{0}^{1}\|f_r\|^q_X\om(r)\,dr<\infty\right\},\quad 0<q<\infty,
	$$
	and
	$$
	X(\infty,\om)=\left\{f\in\H(\D):\|f\|_{X(\infty,\om)}=\sup_{0<r<1}\|f_r\|_X\widehat{\om}(r)<\infty\right\}.
	$$
In particular $\Apq{p}{q} = H^p(q,\om)$. Let us also recall that the classical Bloch space $\B$ consists of the functions
$f\in\H(\D)$ such that $\|f\|_{\B}=|f(0)|+\sup_{z\in\D}(1-|z|^2)|f'(z)|<\infty$.

\begin{theorem}\label{th: dualidad general}
		Let $\om\in \DD$ and $0<p,q<\infty$. Then,  the following assertions hold:
		\begin{enumerate}
			\item[(i.a)] If $1< p <\infty$ and $1<q<\infty$, then  $\left ( \Apq{p}{q} \right )^\ast \simeq \Apq{p'}{q'}$ under the $A^2_\om$-pairing with equivalence of norms;
			\item[(i.b)] If $1< p <\infty$ and  $0<q\leq 1$, then  $\left ( \Apq{p}{q} \right )^\ast \simeq \Apq{p'}{\infty}$ under the $A^2_{\mu_q}$-pairing
			with equivalence of norms, where $\mu_q = \om \omg^{\frac{1}{q}}$;
			\item[(ii.a)] If $p=1$ and $1<q< \infty$, then  $\left ( \Apq{p}{q} \right )^\ast \simeq BMOA(q',\om)$ under the $A^2_\om$-pairing with equivalence of norms;
			\item[(ii.b)] If $p=1$ and $0<q\leq 1$, then  $\left ( \Apq{p}{q} \right )^\ast \simeq BMOA(\infty,\om)$ under the $A^2_{\mu_q}$-pairing
			with equivalence of norms, where $\mu_q = \om \omg^{\frac{1}{q}}$;
			\item[(iii.a)] If $0<p<1$ and $1<q<\infty$, then $\left ( \Apq{p}{q} \right )^\ast \simeq \B(q',\om)$ under the pairing 
			$$
			\langle f,g\rangle_{\om,p} = \lim_{r\to 1^-} \sum_{n=0}^\infty \widehat{f}(n) \overline{\widehat{g}(n)}\om_{2n+1}
\left(\nu_{\frac{1}{p}-2} \right)_{2n+1} r^n,
			$$
			with equivalence of norms;
			\item[(iii.b)] If $0<p<1$ and $0<q\leq 1$, then $\left ( \Apq{p}{q} \right )^\ast \simeq \B(\infty,\om)$ under the pairing  
			$$
			\langle f,g\rangle_{\om,p,q} = \lim_{r\to 1^-} \sum_{n=0}^\infty \widehat{f}(n) \overline{\widehat{g}(n)}(\mu_q)_{2n+1}\left(\nu_{\frac{1}{p}-2} \right)_{2n+1} r^n,
			$$
			with equivalence of norms, where $\mu_q = \om  \omg^{\frac{1}{q}}$.
		\end{enumerate}
	\end{theorem}
We recall that in the statement of Theorem~\ref{th: dualidad general},   $\nu_{\frac{1}{p}-2}(z) =(\frac{1}{p}-1) \left (1-\abs{z}^2 \right )^{\frac{1}{p}-2}$, $0<p<1$.

Useful descriptions of $\left ( A^p_\om \right )^\ast$, $\om\in\DD$, are known (see \cite[Theorem~4 and proof of Theorem~7]{advances} for the case $1\le p<\infty$   and   \cite[Theorem~1]{PeralaRattyaAASF2017} for the case $0<p<1$). Moreover, handy descriptions of the	dual spaces of mixed normed spaces induced  standard weights are known, see \cite{AhJet} and the references therein. However, our descriptions may differ in form from those obtained in the standard case; this is a deliberate choice to ensure the resulting spaces in our characterizations remain as simple and manageable as possible.

The proof of Theorem~\ref{th: dualidad general} is rather involved and relies on a large number of technical lemmas as well as known results from the literature. It will be divided, according to the value of $p$, into the proofs of Proposition~\ref{prop: dualidad 1<p}, Proposition~\ref{prop: dualidad 1<q y p=1}, and Proposition~\ref{prop: dualidad p<1}. Among these, the proof of Proposition~\ref{prop: dualidad p<1} is the most technical, although all three follow a similar general pattern, which we briefly outline below.
The first observation is that the dual of an admissible quasi-Banach space $X$ (see Subsection~\ref{sec21} for a precise definition), in which the polynomials are dense, can be identified with the space of coefficient multipliers from $X$ to $H^\infty$ under the $H^2$-pairing. Secondly, if $X$ is a Banach admissible space or a Hardy space, then there exist $K>1$ and a sequence of polynomials $\{P_n\}_{n=0}^\infty$ such that
$\|f\|_{X(\infty,\om)}\asymp\sup_{n\in\N\cup\{0\}}K^{-n}\|P_n\ast f\|_{X}$ and for $0<q<\infty$ $\|f\|_{X(q,\om)}\asymp\sum_{n=0}^\infty K^{-n}\|P_n\ast f\|^q_{X}$, where $\ast$ denotes the convolution. These equivalent norms will be used in what follows.
Moreover, when $0<p\le 1$, one uses the fact that for any admissible Banach space $X$,  $\alpha>0$ and
$\mu=\om\omg^{\alpha-1}$ the operator $I^\mu g(z)=\sum_{k=0}^\infty\widehat{g}(k)\mu_{2k+1}z^{k}$ satisfies $\|I^\mu P_n\ast f\|_{X}\asymp K^{-\alpha n}\|P_n\ast f\|_{X}$, see Proposition~\ref{prop: estimacion norma poli} below. At this step, the hypothesis $\omega\in\DD$ and the uniform boundedness of Cesàro means on $X$ are used in an essential way.
Finally, known descriptions of the dual spaces $(H^p)^\star$ (see \cite[Theorem 9.8]{zhu} for $1<p<\infty$, \cite{girela} for $p=1$, and \cite{Zhu2} for $0<p<1$) allow us to identify the coefficient multipliers from $H^p$ to $H^\infty$. Some further technical calculations involving the operator $I^\mu$ then yield the desired isomorphism.

We will apply Theorem~\ref{th: dualidad general} to study the boundedness of the Bergman projection 
$$
	P_\om(f)(z) = \int_\D f(\zeta)\overline{B^\om_z(\zeta)}\om(\zeta)dA(\zeta),\quad z\in\D,
	$$
	induced by a radial weight $\omega$ on weighted mixed norm Lebesgue  spaces $L^{p,q}_\om$. Here and on the following, $B^\om_z$, $z\in \D$, are the reproducing kernels of the weighted Bergman space  $A^2_{\om}$.  The classical Bergman kernel $B^\alpha_z$ has the very useful formula 
	$
	B^\alpha_z(\zeta)=\frac1{(1-\overline{z}\zeta)^{2+\alpha}}, \quad z,\z\in \D.
	$
However, if $\om$ is only assumed to be a radial weight, then the Bergman reproducing kernel $B^\omega_z$  does not have such a neat explicit expression, and this is precisely one of the main difficulties to tackle the aforementioned question. Consequently, we are forced to work with the identity $B^\om_z(\z)=\sum_{n=0}^\infty\frac{\left(\overline{z}\z\right)^n}{2\om_{2n+1}}$ $z,\z\in\D$.
	The boundedness of projections  on various function spaces is a compelling topic which has attracted a considerable amount of attention during the last decades. This is not only due to the mathematical difficulties the question raises, but also to its numerous applications in important questions in operator theory such as duality relationships, Littlewood-Paley inequalities, and the famous Sarason's conjecture on Bergman spaces
 \cite{AlPoRe, Gad,Jet,PelRatproj,advances,Shi,Zhu2}. In our case, as a byproduct of  Theorem~\ref{th: dualidad general} (i.a), we establish the boundedness of the Bergman projection on $L^{p,q}_\om$ for $1<p,q<\infty$ and $\om \in \DD$. 

	\begin{theorem}\label{coro: acotacion proyeccion}
		Let $\om\in \DD$ and $1<p,q<\infty$. Then $P_\om:L^{p,q}_\om \to L^{p,q}_\om$ is bounded.
	\end{theorem}

	Obtaining the boundedness of an analytic projection as a consequence of the identification of  the dual of certain space of analytic functions differs from the classical approach. In the case of an standard weight $\omega$ \cite{Gad,Zhu2} and even for other weighted versions of mixed norm spaces \cite{Jet,Shi}, the classical approach consists in obtaining a description of the dual spaces of $A^{p,q}_\om$  as a byproduct of the boundedness of some Bergman projection. 
 Nevertheless, we proceed in the opposite direction, as the classical approach to get the boundedness of the Bergman projection relies fundamentally on the fact that it is sufficient to establish the boundedness of the maximal Bergman projection
$$
	P^+_{\om}(f)(z) = \int_\D f(\zeta) \left|B^\om_z(\zeta)\right| \om(\zeta) \,dA(\zeta),\quad z\in\D.
	$$
 The next result shows that this  is not the case for mixed norm spaces induced by $\om\in\DD$.

	\begin{theorem}\label{th: acotacion P+}
		Let $\om\in \DD$ and $1<p,q<\infty$. Then $P^+_\om:L^{p,q}_\om \to L^{p,q}_\om$ is bounded if and only if $\om \in \DDD$.
	\end{theorem}

Theorem~\ref{coro: acotacion proyeccion} together with Theorem~\ref{th: acotacion P+} shows that there exists a
 cancellation of the Bergman reproducing kernel, which plays an essential role in the boundedness of the Bergman projection $P_\om$ on weighted mixed norm Lebesgue  space $L^{p,q}_\om$ induced by an upper doubling radial weight $\omega$.
	As a consequence, we deduce that whenever $\om\in\DD\setminus\Dd$ and $1<p,q<\infty$, the Bergman projection $P_\omega$ is bounded on $L^{p,q}_\om$, whereas the maximal Bergman projection is not. The proof of Theorem~\ref {th: acotacion P+} relies on an appropiate modification of Schur's test for this class of spaces, together with the following  precise estimate for the integral means of order one of the Bergman reproducing kernels.

\begin{proposition}\label{pr:kernels-p=1}
	Let $\om\in \DD$ and $N\in \N\cup\{0\}$. Then,
	$$M_1(r,(B^\om_a)^{(N)})\asymp 1+\int_{0}^{r|a|}\frac{dt}{(1-t)^{N+1}\omg(t)}, \quad r\in (0,1), a \in \D.$$
\end{proposition}

Proposition~\ref{pr:kernels-p=1} follows from \cite[Theorem~1]{PelRatproj}, however we will provide in Section 4 a simplified and direct proof. 
 		The rest of the paper is organized as follows. Section 2 introduces all the necessary notation and preliminary results on radial weights, dual spaces and decomposition norms. Section 3 contains the proofs of all the results leading to Theorem 1.1. Finally, Section 4 is devoted to the study of the Bergman projection $P_\omega$, in particular,  to the proofs of Theorem~\ref{coro: acotacion proyeccion} and Theorem~\ref{th: acotacion P+}. 
 	
 	Finally, we introduce the following notation, we will use $a\lesssim b$ if there exists a constant $C=C(\cdot)>0$ such that $a\leq Cb$, and $a\gtrsim b$ is understood in an analogous manner. In particular, if $a\lesssim b$ and $a\gtrsim b$, then we write $a\asymp b$ and say that $a$ and $b$ are comparable.
\section{Background}

In this section, we present some notation and gather a good number of results which will be  used in the proof of Theorem~\ref{th: dualidad general}.

\subsection{Radial weights}

For $\beta>0$ and $\om$ a radial weight we denote $\om_{ \{ \beta \} }(r)=(1-r^2)^\beta\om(r)$. 
We will use the following descriptions of the class $\DD$, proved in \cite[Lemma~2.1]{PelSum14} and \cite[(1.3)]{advances}. See also \cite[Lemma A]{PeldelaRosa22}.

\begin{letterlemma}
	\label{lemma: caract dgorro}
	Let $\om$ be a radial weight. Then, the following statements are equivalent:
	\begin{itemize}
		\item[\rm(i)] $\om \in \DD$;
		\item[\rm(ii)] There exist $C=C(\om)\geq 1$ and $\a_0=\a_0(\om)>0$ such that
		$$ \omg(s)\leq C \left(\frac{1-s}{1-t}\right)^{\a}\omg(t), \quad 0\leq s\leq t<1,$$
		for all $\a\geq \a_0$;
		\item[\rm(iii)] There exist $C_1=C_1(\om)>0$  and $C_2=C_2(\om)>0$ such that
		$$ C_1\om_x\le \omg\left(1-\frac{1}{x}\right)\le  C_2 \om_x ,\quad x \in [1,\infty);$$
		\item[\rm(iv)] There exists $C(\om)>0$ such that $\om_x\le C \om_{2x}$, for any $x\ge 1$;
		\item[\rm(v)] There exists $C(\om)>0$ and $\eta(\om)>0$ such that
		$$
		\om_x \leq C \left (\frac{y}{x} \right )^\eta \om_y;
		$$
		\item[\rm(vi)] For any $\beta>0$ (equivalently for some) there exists $C=C(\b,\om)>0$ such that
		$$ x^{\beta}\left( \om_{\{\beta\}}\right)_x\le C\om_x, \quad x\ge 1.$$
	\end{itemize}
\end{letterlemma}

We will also use  the following characterization of the class $\Dd$, see 
\cite[Lemma 4]{MorPeldelaRosa23}.
\begin{letterlemma}
	\label{caract. D check}
	Let $\om$ be a radial weight. Then, the
	following statements are equivalent:
	\begin{itemize}
		\item[\rm(i)] $\om	\in \Dd$;
		\item[\rm(ii)] For each (or some) $\g >0$ there exists $C=C(\g, \om)>0$ such that
		$$ \int_{0}^r \frac{ds}{\omg(s)^{\g}(1-s)} \leq \frac{C}{ \omg(r)^{\g}},\quad 0\leq r < 1.$$
	\end{itemize}
\end{letterlemma}

We will use the following technical result.
\begin{lemma}\label{lema: peso alfa}
	Let $\om\in \DD$ and $\a>0$. Then $\mu = \om \omg^{\a-1}\in \DD$ and $\mu_x \asymp \om_x^\a$, for all $x\geq 1$.
\end{lemma}
\begin{proof}
	Since $\widehat{\mu}=\frac{1}{\alpha}\omg^\alpha$, it is clear that $\mu\in \DD$. Moreover,  it follows from  Lemma~\ref{lemma: caract dgorro} (iii) that 
	$\mu_x \asymp \om_x^\a$, for all $x\geq 1$. This finishes the proof.
\end{proof}

\subsection{ Dual spaces of admissible Banach spaces.}\label{sec21}
A vector space $X\subset\H(\D)$ endowed with a quasi-norm $\|\cdot\|_X$ is admissible if
\begin{enumerate}
	\item $\|\cdot\|_X$-convergence implies the uniform convergence on compact subsets of $\D$;
	\item for each $r>1$, the uniform convergence on compact subsets of $D(0,r)$ implies $\|\cdot\|_X$-convergence;
	\item There exists $C=C(X)>0$ such that   the function $f_\zeta$ defined by $f_\zeta(z)=f(\zeta z)$ for all $z\in\D$ satisfies $\|f_\z\|_X\le C\|f\|_X$ for all $f\in X$ and $\zeta\in\overline{\D}$.
\end{enumerate}
Examples of admissible Banach spaces are: the Hardy spaces $H^p$, $p\ge 1$;  the weighted mixed norm spaces  $A^{p,q}_\om$, $p,q\ge 1$; the classical Bloch space $\B$ and $\BMOA$ \cite{Duren,girela,Zhu2}.

The Hadamard product (or coefficient multiplication)  of $f(z)=\sum_{k=0}^\infty \widehat{f}(k)z^k\in\H(\D)$ and $g(z)=\sum_{k=0}^\infty \widehat{g}(k)z^k\in\H(\D)$ is defined as follows (whenever it makes sense)
$$
(f\ast g)(z)=\sum_{k=0}^\infty\widehat{f}(k)\widehat{g}(k)z^k,\quad z\in\D.
$$
For quasi-normed spaces $X,Y\subset \H(\D)$, let $[X,Y]=\set{g\in \H(\D):f\ast g\in Y \text{ for all }f\in X}$ denote the space of coefficient multipliers from $X$ to $Y$ equipped with the norm
$$
\nm{g}_{[X,Y]} = \sup\set{\nm{f\ast g}_Y : f\in X, \nm{f}_X \leq 1}.$$

The next result follows from \cite[Proposition~1.3]{pavlovic}.

\begin{letterproposition}\label{pr:dual-admisible}
	Let $X\subset \H(\D)$ a quasi-normed admissible vector space such that the polynomials are dense in $X$. Then,  $X^\star\simeq [X, H^\infty]$ via the $H^2$-pairing with equivalence of norms. 
\end{letterproposition}
Ii $\om$ is a radial weight,  an standard argument shows that the polynomials are dense in  the mixed norm spaces $\Apq{p}{q}$ for $0<p,q<\infty$. In the proof of Theorem~\ref{th: dualidad general} we will apply Proposition~\ref{pr:dual-admisible} to  $X=\Apq{p}{q}$ for $0<p,q<\infty$.

\subsection{Decomposition norm results on generalized mixed norm spaces $X(q,\om)$}

One of the main ideas used in the proof of Theorem~\ref{th: dualidad general} consists on using a decomposition norm result for  $X(q,\om)$, where $0<q\le \infty$ and $X$ is an admissible Banach space or any Hardy space $H^p$, $0<p<\infty$.   For $1<p<\infty$, which is the range such that the Riesz projection is bounded on
$L^p(\partial \D)$,   these results are well-known for mixed norm spaces
$\Apq{p}{q}$ induced by standard weights \cite[p. 120]{Pabook}  and they have been generalized to the setting of upper radial doubling weights \cite[Theorem~3.4]{advancedcourse} (see also \cite[Theorem~4]{PelRathg}) or  to the case $0<p\le 1$, see  Proposition~\ref{pr:cesaro} below. However,   the previously mentioned results are not sufficient for our purpose, since they are needed for all the spaces  $X(q,\om)$ appearing in the statement of Theorem~\ref{th: dualidad general}. For this reason, we recall some notation and results.
An increasing sequence $\{\lambda_n\}_{n=0}^\infty$ of natural numbers is lacunary if there exists $C>1$ such that $\lambda_{n+1}\ge C\lambda_{n}$ for all $n\in\N\cup\{0\}$. The next result follows from \cite[Theorem~2.1]{pavlovic} and its proof.

\begin{lettertheorem}\label{th:Xpolynomials}
	Let $\{\lambda_n\}_{n=0}^\infty$ be a lacunary sequence and $N\in\N$. Then there exists a sequence of polynomials $\{P_n\}_{n=0}^\infty$ and constants $C_1=C_1(\{\lambda_n\},N)>0$ and $C_2=C_2(\{\lambda_n\},N)>0$ such that
	\begin{equation}\label{eq:P1}
		f=\sum_{n=0}^\infty P_n\ast f,\quad f\in \H(\D);
	\end{equation}
	
	\begin{equation}\label{eq:P2}
		\widehat{P_n}(j)=0,\quad j\notin \left[\lambda_{n-1},\lambda_{n+N}\right),\quad \lambda_{-1}=0,\quad n\in\N\cup\{0\};
	\end{equation}
	
	\begin{equation}\label{eq:P3}
		\left\|\sum_{n=0}^k P_n*f\right\|_X\le C_1 \|f\|_X, \quad f \in \H(\D),\quad k\in\N\cup\{0\};
	\end{equation}
	
	\begin{equation}\label{eq:P4}
		\left\|P_n*f\right\|_X\le C_2 \|f\|_X, \quad f \in \H(\D),\quad n\in\N\cup\{0\};
	\end{equation}
	
	\begin{equation}\label{eq:P5}
		\lim_{k\to\infty}\left\|f-\sum_{n=0}^k P_n\ast f\right\|_X=0, \quad f \in X^0,
	\end{equation}
	for all admissible Banach space $X$ and for all $X=H^p$ with $p>\frac{1}{N+1}$.
\end{lettertheorem}

Now,  for each $K>1$ and $n\in\N\cup\{0\}$, let be
\begin{equation}\label{eq: definicion rn}
	r_n=r_n(\om,K)=\inf\{  r\in [0,1):	\omg(r) = \omg(0)K^{-n} \}.
\end{equation} 
Clearly, $\{r_n\}_{n=0}^\infty$ is an increasing sequence of distinct points on $[0, 1)$ such that $r_0 = 0$ and $r_n\to1^-$, as $n \to \infty$. For $x\in[0,\infty)$, let $E(x)$ denote the integer such
that $E(x) \leq x < E(x)+1$, and set 

\begin{equation}\label{eq:Mn}
	M_n = M_n(\om,K) = E\left (\frac{1}{1-r_n} \right)
\end{equation}	

for short. 
We will use the following thechnical result.
\begin{lemma}\label{le:Mnstar}
	Let $\om\in\DD$. Then, there exists $K=K(\om)>1$ such that the sequence $\{M_n\}_{n=0}^\infty$ defined in \eqref{eq:Mn}
	is lacunary.
\end{lemma}
\begin{proof}
	Take $\alpha_0>0$ and $C>0$ such that Lemma~\ref{lemma: caract dgorro} (ii) holds. Now choose $K>\max \left\{ \frac{\omg(0)}{\omg\left( \frac{1}{2}\right)}, C3^{\alpha_0} \right\}.$ Then, 
	\begin{equation*}\begin{split}
			M_{n+1} & \ge \frac{r_{n+1}}{1-r_{n+1}} \\ & \ge \frac{r_{n+1}}{1-r_{n}}\left( \frac{K}{C}\right)^{\frac{1}{\alpha_0}}
			> 3\frac{r_{n+1}}{1-r_{n}} \ge 3 r_1M_n\ge \frac{3}{2}M_n, 
			\quad n\in \N\cup\{0\}.
	\end{split}\end{equation*}
	where in the last inequality we have used that $r_1\ge \frac{1}{2}$ by the election of $K$. This finishes the proof.
\end{proof}

Bearing in mind Lemma~\ref{le:Mnstar}, the following decomposition norm result for  $X(q,\om)$  will used in the proof of Theorem~\ref{th: dualidad general}, see   \cite[Theorem~25]{advances} for a proof.

\begin{lettertheorem}\label{th: descomposicion pavlovic}
	Let $\om \in \DD$, $0 < q  \leq \infty$ and $N \in \N \cup \{0\}$. Let $K = K(\om) > 1$ such that $\{M_n\}_{n=0}^\infty$ is a lacunary sequence, and let $\{P_n\}_{n=0}^\infty$ be the sequence of polynomials associated to N and $\{M_n\}_{n=0}^\infty$ via Theorem~\ref{th:Xpolynomials}. Then
	\begin{equation}\label{eq: descomposicion pq general}
		\sum_{n=0}^\infty K^{-n}\nm{P_n\ast f}_X^q \asymp \nm{f}_{X(q,\om)}^q,\quad 0<q<\infty,\quad f\in \H(\D)
	\end{equation}
	and
	\begin{equation}\label{eq: descomposicion pinf general}
		\sup_{n\in \N\cup\{0\}} K^{-n}\nm{P_n\ast f}_X \asymp \nm{f}_{X(\infty,\om)},\quad f\in \H(\D)
	\end{equation}
	for all admissible Banach spaces $X$ and for all $X = H^p$ with $p>\frac{1}{N+1}$.
\end{lettertheorem}

\subsection{A technical result on certain subclasses of vector valued Lebesgue spaces}
For a quasi-normed space $X\subset\H(\D)$, $s\in\mathbb{R}$, $0<q\le\infty$ and a sequence of polynomials $\{P_n\}_{n=0}^\infty$, let us conisder 
\begin{equation}\label{eq:lqXPn}
	\ell^q_s(X,\{P_n\})
	=\left\{f\in\H(\D):\|f\|^q_{\ell^q_s(X,\{P_n\})}=\sum_{n=0}^\infty\left(2^{-ns}\|P_n*f\|_X\right)^q<\infty\right\}, \quad 0<q<\infty,
\end{equation}
and
\begin{equation}\label{eq:linftyXPn}
	\ell^\infty_s(X,\{P_n\})
	=\left\{f\in\H(\D):\|f\|_{\ell^\infty_s(X,\{P_n\})}=\sup_n\left(2^{-ns}\|P_n*f\|_X\right)<\infty\right\}.
\end{equation}

It is worth mentioning that if $K$, $X$ and $\{P_n\}_{n=0}^\infty$ are those of Theorem~\ref{th: descomposicion pavlovic}, then  

\begin{equation}\begin{split}\label{eq:s}
		X(q,\om) &=\ell^q_{\frac{s}{q}} (X,\{P_n\}),\quad\text{ where $0<q< \infty$ and $s=\log_2 K$,}
		\\ X(\infty,\om) &=\ell^\infty_{s} (X,\{P_n\}),\quad\text{ where  $s=\log_2 K$}.
\end{split}\end{equation}
From now on,   for any $0<q\le \infty$ we denote by $q'$ its conjugate exponent defined as follows $$
q' = \begin{cases}
	\infty & \text{ if  } 0<q\leq 1, \\
	\frac{q}{q-1}& \text{ if } 1<q<\infty, \\
	1 & \text{ if } q= \infty .
\end{cases} 
$$
The next result follows from \cite[Theorem~5.4]{pavlovic}.

\begin{lettertheorem}\label{th:pavlovic-bloques}
	Let $\{\lambda_n\}_{n=0}^\infty$ be a lacunary sequence and $N\in\N$.  Let $\{P_n\}_{n=0}^\infty$ be the sequence of polynomials associated to $N$ and $\{\lambda_n\}_{n=0}^\infty$ via Theorem~\ref{th:Xpolynomials}.  If  $s\in\mathbb{R}$, $0<q<\infty$  and  $X\subset\H(\D)$ is an admissible  Banach space or $H^p$ with $p>\frac{1}{N+1}$, then 
	$$\left[ \ell^q_s(X,\{P_n\}), H^\infty\right]= \ell^{q'}_{-s}([X, H^\infty],\{P_n\})$$
	with equivalence of norms.
\end{lettertheorem}

\section{Proof of Theorem~\ref{th: dualidad general}.}
We will split the proof of Theorem~\ref{th: dualidad general} into several cases according to the values of the parameteres $p$ and $q$. In all of them, given  a radial weight $\om$ we will consider the action of the operator
$I^\om:\H(\D)\to \H(\D)$ defined by
\begin{equation}\label{eq:Iw}
	I^\om g(z) = \sum_{n=0}^\infty \widehat{g}(n)\om_{2n+1}z^n,\quad z\in \D.
\end{equation}

\subsection{Proof of Theorem~\ref{th: dualidad general}. Cases (i.a) and (i.b).}

\begin{proposition}\label{prop: dualidad 1<p}
	Let $\om\in \DD$, $1<p<\infty$ and $0<q<\infty$. Then, $(\Apq{p}{q})^\ast \simeq \Apq{p'}{q'}$ under the $A^2_\om$-pairing with equivalence of norms.
\end{proposition}
\begin{proof}
	Take  $K(\om)>1$ such that $\{M_n\}_{n=0}^\infty$ is a lacunary sequence, and  let $\{P_n\}_{n=0}^\infty$ be the sequence of polynomials associated to $N=0$ and $\{M_n\}_{n=0}^\infty$ via Theorem~\ref{th: descomposicion pavlovic}. Then  if $s=\log_2(K)$,
	$\Apq{p}{q}=\ell^q_{\frac{s}{q}} (H^p, \{P_n\})$. This identity together with 
	Proposition~\ref{pr:dual-admisible} yields
	$$
	(\Apq{p}{q})^\ast \simeq [\Apq{p}{q},H^\infty] \simeq \left[ \ell^q_{\frac{s}{q}} (H^p, \{P_n\}), H^\infty \right ]
	$$
	under the $H^2$-pairing with equivalence of norms.  Therefore,  by Theorem~\ref{th:pavlovic-bloques} and the fact $(H^p)^\ast\simeq[H^p,H^\infty] = H^{p'}$  via the $H^2$-pairing \cite[Theorem~9.8]{Zhu2}, we deduce that
	\begin{equation}\label{eq:c11}
		(\Apq{p}{q})^\ast \simeq  \ell^{q'}_{-\frac{s}{q}} \left([H^p,H^\infty],  \{P_n\} \right) = \ell^{q'}_{-\frac{s}{q}} \left ( H^{p'},  \{P_n\} \right ),
	\end{equation}
	under the $H^2$-pairing with equivalence of norms.  Now, bearing in mind that the sequence $\{\om_{2n+1}\}_{n=0}^\infty$ is non-increasing, \cite[Lemma E]{PelRathg} and Theorem~\ref{th:Xpolynomials},  there exists a constant $C>0$ such that
	$$
	C\om_{M_{n+1}}\nm{P_n\ast  g}_{H^{p'}} \leq \nm{ I^\om P_n\ast  g}_{H^{p'}} \leq  C\om_{M_{n}}\nm{P_n\ast  g}_{H^{p'}}, \quad n\in \N\cup\{0\}, \quad g \in \H(\D).
	$$
	This combined to Lemma~\ref{lemma: caract dgorro} (iii) and the definition of $M_n$ gives that
	\begin{equation*}
		K^{-n}\nm{P_n\ast  g}_{H^{p'}} \asymp \nm{ I^\om P_n\ast  g}_{H^{p'}}, \quad n\in \N\cup\{0\}, \quad g \in \H(\D).
	\end{equation*}
	So if $1<q<\infty$,  by Theorem~\ref{th: descomposicion pavlovic}
	\begin{equation*}
		\begin{split}
			\nm{g}_{\Apq{p'}{q'}}^{q'} &
			\asymp 
			\sum_{n=0}^{\infty} K^{-n} \nm{\ P_n\ast  g}_{H^{p'}}^{q'} \asymp \sum_{n=0}^{\infty} K^{-n + nq'} \nm{ I^\om   P_n\ast  g }_{H^{p'}}^{q'} \\ &= \sum_{n=0}^{\infty} K^{n\frac{q'}{q}} \nm{    P_n\ast  I^\omega g }_{H^{p'}}^{q'}
			= \nm{I^\om g}_{\ell^{q'}_{-\frac{s}{q}}(H^{p'},\{ P_n \})}^{q'}, \quad g\in \H(\D),
		\end{split}
	\end{equation*}
	and hence $I^\om : \Apq{p'}{q'} \to \ell^{q'}_{-\frac{s}{q}}(H^{p'},\{ P_n\})$ is isomorphic and onto. This turns the equivalence $(\Apq{p}{q})^\ast \simeq \ell^{q'}_{-\frac{s}{q}}(H^{p'},\{ P_n \})$ via $H^2$-pairing, obtained in \eqref{eq:c11},  to $(\Apq{p}{q})^\ast \simeq \Apq{p'}{q'}$ via $A^2_\om$-pairing. This finishes the proof in the case $1<q<\infty$.
	
	Next, if $0<q\le 1$
	we consider the weight $\mu_q =\om\omg^{\frac{1}{q}}$.  Then by \cite[Lemma E]{PelRathg}, Lemma~\ref{lemma: caract dgorro} (iii) and Lemma~\ref{lema: peso alfa} we obtain the equivalence
	\begin{equation*}
		K^{-n\left(1+\frac{1}{q}\right) }\nm{P_n\ast  g} \asymp \left(\mu_q\right)_{M_{n}}\nm{P_n\ast  g} \asymp  \nm{ I^{\mu_q} P_n\ast  g}, \quad n\in \N\cup\{0\}, \quad g \in \H(\D).
	\end{equation*}
	Therefore,  by Theorem~\ref{th: descomposicion pavlovic}
	\begin{equation*}
		\begin{split}
			\nm{g}_{\Apq{p'}{\infty}} &\asymp
			\sup_{n\in \N \cup \{0\}} K^{-n}\nm{ P_n\ast g}_{H^{p'}}\\
			&\asymp \sup_{n\in \N \cup \{0\}} K^{\frac{n}{q}}\nm{ I^{\mu_q} P_n\ast g}_{H^{p'}} =\nm{I^{\mu_q} g}_{\ell^{\infty}_{-\frac{s}{q}}(H^{p'},\{ P_n \})}, \quad g\in \H(\D),
		\end{split}
	\end{equation*}
	so the operator $I^{\mu_q}: \Apq{p'}{\infty} \to \ell^{\infty}_{-\frac{s}{q}}(H^{p'},\{ P_n \})$ is isomorphic and onto. This turns the equivalence $(\Apq{p}{q})^\ast \simeq \ell^{\infty}_{-\frac{s}{q}}(H^{p'},\{ P_n \})$ via $H^2$-pairing, obtained in \eqref{eq:c11},  to $(\Apq{p}{q})^\ast \simeq \Apq{p'}{\infty}$ under the $A^2_{\mu_q}$-pairing, concluding the proof.
\end{proof}

\subsection{Proof of Theorem~\ref{th: dualidad general}. Cases (ii.a) and (ii.b).}
In order to deal with the proof of Theorem~\ref{th: dualidad general} for  $p=1$ we need to introduce some extra notation and to prove some   results. 
For $n\in\N$, let $K_n$ denote the Ces\'aro kernel defined by $K_n(z)=\sum_{j=0}^n\left( 1-\frac{j}{n+1}\right)z^j$ for all $z\in\D$, and write $\sigma_n f=K_n*f$ for the Ces\'aro means of $f$.

The following result \cite[Theorem~2.2]{pavlovic} extends the classical theorem of Hardy and Littlewood on the boundedness of the  Ces\'aro means (or (C,1) means) on $H^p$ for $\frac{1}{2}<p<\infty$ \cite{HL}.

\begin{lettertheorem}\label{th:cesaro}
	Let $X\subset \H(\D)$ be an admissible Banach space. Then, there exists $C>0$ such that 
	$$\| \sigma_n f\|_X\le C \|f\|_{X}, \quad n\in \N\cup\{0\}, \quad f \in \H(\D).$$
\end{lettertheorem}

Theorem~\ref{th:cesaro} is a key tool in the proof of the next technical result.

\begin{proposition}\label{prop: estimacion norma poli}
	Let $\om\in\DD$, $N\in \N$, $\a>0$ and $\mu = \om \omg^{\a-1}$. Let $K=K(\om)>1$ such that $\{M_n\}_{n=0}^\infty$ is a lacunary sequence, and let $\{P_n\}_{n=0}^\infty$ be the sequence of polynomials associated to $N$ and $\{M_n\}_{n=0}^\infty$ via  Theorem~\ref{th:Xpolynomials}. Then 
	\begin{equation}
		\begin{split}\label{eq:Io}
			\nm{I^\mu (P_n\ast g)}_{X}\asymp \mu_{M_n} \nm{P_n*g}_{X}\asymp  K^{-\a n} \nm{P_n*g}_{X},\quad n\in\N\cup\{0\},\quad g\in\H(\D),
		\end{split}
	\end{equation}
	for all admissible Banach space $X$.
\end{proposition}
\begin{proof}
	We follow the argument of  the proof of \cite[Theorem 26]{advances} where it is proved the case $\alpha=1$. Theorem~\ref{th:Xpolynomials} yields
	\begin{equation*}
		I^\mu(P_n\ast g)(z)= \sum_{j= M_{n-1}}^{M_{n+N}} \mu_{2j+1} \widehat{P_n}(j)\widehat{g}(j)z^j = \sum_{j= M_{n-1}}^{M_{n+N}} \mu_{2j+1}
		g^j_n(z),\quad z\in\D,\quad n\in\N\cup\{0\},
	\end{equation*}
	where $g(z)=\sum_{j=0}^\infty\widehat{g}(j)z^j$ and $g^j_n(z)=\widehat{P_n}(j)\widehat{g}(j)z^j$ for all $n\in\N\cup\{0\}$. Write $G^j_n=\sum_{m=0}^j g^m_n$ for short. Then $G^k_n\equiv0$ for all $k< M_{n-1}$ by Theorem~\ref{th:Xpolynomials}, and hence a summation by parts gives
	\begin{equation}
		\begin{split}\label{eq:Io1}
			I^\mu(P_n*g)
			&=\sum_{j=M_{n-1}}^{M_{n+N}-1}(\mu_{2j+1}-\mu_{2j+3})G_n^j+\mu_{2M_{n+N}+1}\left(P_n*g\right),\quad n\in\N\cup\{0\}.
		\end{split}
	\end{equation}
	Now $\sigma_k(P_n*g)\equiv0$ for $k< M_{n-1}$, and hence a direct calculation yields
	\begin{equation*}
		\begin{split}
			&\sum_{j=M_{n-1}}^{M_{n+N}-1}(\mu_{2j+1}-\mu_{2j+3})G_n^j\\
			&=\sum_{j=M_{n-1}}^{M_{n+N}-1}(\mu_{2j+1}-\mu_{2j+3})
			\left[(j+1)\sigma_j(P_n\ast g)-j\sigma_{j-1}(P_n\ast g)\right]\\
			&=\sum_{j=M_{n-1}}^{M_{n+N}-2}\left[(\mu_{2j+1}-\mu_{2j+3})-(\mu_{2j+3}-\mu_{2j+5})\right]
			(j+1)\sigma_j (P_n*g)\\
			&\quad+(\mu_{2M_{n+N}-1}-\mu_{2M_{n+N}+1})M_{n+N}\sigma_{M_{n+N}-1}(P_n*g),\quad n\in\N\cup\{0\},
		\end{split}
	\end{equation*}
	which together with \eqref{eq:Io1} gives
	\begin{equation}
		\begin{split}\label{eq:Io2}
			I^\mu(P_n\ast g)
			&=\sum_{j=M_{n-1}}^{M_{n+N}-2}\left[(\mu_{2j+1}-\mu_{2j+3})-(\mu_{2j+3}-\mu_{2j+5})\right](j+1)\sigma_j (P_n\ast g)\\
			&\quad+(\mu_{2M_{n+N}-1}-\mu_{2M_{n+N}+1})M_{n+N}\sigma_{M_{n+N}-1}(P_n*g)\\
			&\quad+\mu_{2M_{n-1}+1}(P_n*g),\quad n\in\N\cup\{0\}.
		\end{split}
	\end{equation}
	Next, the $X$-norm of the three terms appearing on the right hand side of \eqref{eq:Io2} will be estimated separately. Firstly,  Theorem~\ref{th:cesaro},  Lemma~\ref{lema: peso alfa}, Lemma~\ref{lemma: caract dgorro}~(iii) and the definition of  $M_n$ give that
	\begin{equation}
		\begin{split}\label{eq:Io3}
			&\left\|\sum_{j=M_{n-1}}^{M_{n+N}-2}\left[(\mu_{2j+1}-\mu_{2j+3})-(\mu_{2j+3}-\mu_{2j+5})\right]
			(j+1)\sigma_j (P_n*g)\right\|_X\\
			&\le\sum_{j=M_{n-1}}^{M_{n+N}-2}
			\left[(\mu_{2j+1}-\mu_{2j+3})-(\mu_{2j+3}-\mu_{2j+5})\right]
			(j+1)\left\|\sigma_j(P_n*g)\right\|_X\\
			&\lesssim \left\|P_n*g\right\|_X\sum_{j=M_{n-1}}^{M_{n+N}-2}\left[(\mu_{2j+1}-\mu_{2j+3})-(\mu_{2j+3}-\mu_{2j+5})\right] (j+1)\\
			&=\left\|P_n*g\right\|_X\sum_{j=M_{n-1}}^{M_{n+N}-2}(\mu_{\{2\}})_{2j+1}(j+1)\\
			&=\left\|P_n*g\right\|_X\int_{0}^1 \left(\sum_{j=M_{n-1}}^{M_{n+N}-2}(j+1)r^{2j+1}\right) (1-r^2)^2\mu(r)\,dr\\
			&\le\left\|P_n*g\right\|_X\int_{0}^1 r^{M_{n-1}+1} \left(\sum_{j=0}^\infty(j+1)r^j\right)  (1-r^2)^2\mu(r)\,dr\\
			&\lesssim \left\|P_n*g\right\|_X\int_{0}^1 r^{M_{n-1}+1} \mu(r)\,dr \\
			& \asymp  \left\|P_n*g\right\|_X \mu_{M_n}
			\asymp K^{-\a n}\left\|P_n*g\right\|_X,\quad n\in\N\cup\{0\}.
		\end{split}
	\end{equation}
	Secondly, Theorem~\ref{th:cesaro},  Lemma~\ref{lema: peso alfa}, Lemma~\ref{lemma: caract dgorro}~(iii) and (vi)  and the definition of  $M_n$ give that
	\begin{equation}
		\begin{split}\label{eq:Io4}
			&\left\|(\mu_{2M_{n+N}-1}-\mu_{2M_{n+N}+1})M_{n+N}\sigma_{M_{n+N}-1}(P_n*g)\right\|_X
			\lesssim(\mu_{[1]})_{M_{n+N}}M_{n+N}\left\|P_n*g\right\|_X\\
			&\quad\lesssim\mu_{M_{n+N}}\left\|P_n*g\right\|_X
			\asymp K^{-\a n} \left\|P_n*g\right\|_X,\quad n\in\N\cup\{0\}.
		\end{split}
	\end{equation}
	Therefore, by combining \eqref{eq:Io2}, \eqref{eq:Io3} and \eqref{eq:Io4}, the first inequality in \eqref{eq:Io} follows.
	
	An argument similar to that in \eqref{eq:Io2} gives
	\begin{equation}
		\begin{split}\label{eq:Io5}
			P_n*g
			&=\sum_{j=M_{n-1}}^{M_{n+N}-2}\left[(\mu_{2j+1}^{-1}-\mu_{2j+3}^{-1})-(\mu_{2j+3}^{-1}-\mu_{2j+5}^{-1})\right]
			(j+1)\sigma_j I^\mu(P_n\ast g)\\
			&\quad+(\mu_{2M_{n+N}-1}^{-1}-\mu_{2M_{n+N}+1}^{-1})M_{n+N}\sigma_{M_{n+N}-1}I^\mu(P_n*g)\\
			&\quad+\mu_{2M_{n-1}+1}^{-1} I^\mu(P_n*g),\quad n\in\N\cup\{0\}.
		\end{split}
	\end{equation}
	By Theorem~\ref{th:cesaro},  Lemma~\ref{lema: peso alfa}, Lemma~\ref{lemma: caract dgorro}~(iii)-(v)-(vi)  and the definition of  $M_n$ 
	\begin{equation*}
		\begin{split}
			&\left\|\sum_{j=M_{n-1}}^{M_{n+N}-2}\left[(\mu_{2j+1}^{-1}-\mu_{2j+3}^{-1})-(\mu_{2j+3}^{-1}-\mu_{2j+5}^{-1})\right]
			(j+1)\sigma_j\left(I^\mu(P_n*g)\right)\right\|_X\\
			&\le\left\|I^\mu(P_n\ast g)\right\|_X \sum_{j=M_{n-1}}^{M_{n+N}-2}\left|(\mu_{2j+1}^{-1}-\mu_{2j+3}^{-1})-(\mu_{2j+3}^{-1}-\mu_{2j+5}^{-1})\right|(j+1)\\
			&=\left\|I^\mu(P_n\ast g)\right\|_X \sum_{j=M_{n-1}}^{M_{n+N}-2}\left|\frac{\left(\mu_{\{2\}}\right)_{2j+1}\mu_{2j+5}-\left(\mu_{\{1\}}\right)_{2j+3}\left(\left(\mu_{\{1\}}\right)_{2j+1}+\left(\mu_{\{1\}}\right)_{2j+3}\right)}
			{\mu_{2j+1}\mu_{2j+3}\mu_{2j+5}}\right|(j+1)\\
			&\le \left\|  I^\mu(P_n\ast g)\right\|_X \sum_{j=M_{n-1}}^{M_{n+N}-2}(j+1)
			\left(\frac{(\mu_{\{2\}})_{2j+1}}{\mu_{2j+1}\mu_{2j+3}}+ 2\frac{(\mu_{\{1\}})^2_{2j+1}}{\mu_{2j+1}\mu_{2j+3}\mu_{2j+5}} \right)
			\\
			&\lesssim\left\|  I^\mu(P_n\ast g)\right\|_X \left( \mu_{M_n}\right)^{-2}\sum_{j=M_{n-1}}^{M_{n+N}-2}(j+1)(\mu_{\{2\}})_{2j+1}
			+\left\|  I^\mu(P_n\ast g)\right\|_X \sum_{j=M_{n-1}}^{M_{n+N}-2}\frac{(\mu_{\{1\}})_{2j+1}}{\mu_{2j+3}\mu_{2j+5}}
			\\
			&\lesssim \left\| I^\mu(P_n\ast g)\right\|_X \left(  \left( \mu_{M_n}\right)^{-1}+  \left( \mu_{M_n}\right)^{-2} \sum_{j=M_{n-1}}^{M_{n+N}-2} (\mu_{\{1\}})_{2j+3} \right)
			\\ &			\lesssim \left( \mu_{M_n}\right)^{-1} \left\|I^\mu(P_n*g)\right\|_X \asymp  K^{-\a n}   \left\|I^\mu(P_n*g)\right\|_X ,\quad n\in\N\cup\{0\}.
		\end{split}
	\end{equation*}
	On the other hand, Theorem~\ref{th:cesaro},  Lemma~\ref{lema: peso alfa}, Lemma~\ref{lemma: caract dgorro}(vi)  and the definition of  $M_n$ give that
	\begin{equation*}
		\begin{split}
			&\left\|(\mu_{2M_{n+N}-1}^{-1}-\mu_{2M_{n+N}+1}^{-1})M_{n+N}\sigma_{M_{n+N}-1}I^\mu(P_n*g)\right\|_X\\
			&\lesssim\left\|I^\mu(P_n*g)\right\|_X  \frac{(\mu_{\{1\}})_{2M_{n+N}-1}}{\mu_{2M_{n+N}-1}\mu_{2M_{n+N}+1}}M_{n+N}\\
			&\lesssim\left\|I^\mu(P_n*g)\right\|_X \frac{1}{\mu_{2M_{n+N}+1}}
			\asymp K^{\a n}\left\|  I^\mu(P_n*g)\right\|_X,\quad n\in\N\cup\{0\},
		\end{split}
	\end{equation*}
	which together with \eqref{eq:Io5} gives the right hand side of \eqref{eq:Io}.
\end{proof}

Now we are ready to prove  Theorem~\ref{th: dualidad general} when $p=1$.
\begin{proposition}\label{prop: dualidad 1<q y p=1}
	Let $\om \in \DD$. Then,
	\begin{enumerate}
		\item[\rm (i)] If  $1<q<\infty$, then $(\Apq{1}{q})^\ast\simeq BMOA(q',\om)$ under the $A^2_\om$-pairing with equivalence of norms;
		\item[\rm (ii)] If  $0<q\leq 1$ and $\mu_q = \om \omg^{\frac{1}{q}}$, then $(\Apq{1}{q})^\ast \simeq BMOA(\infty,\om)$ under the $A^2_{\mu_q}$-pairing with equivalence of norms.
		
	\end{enumerate}

\end{proposition}
\begin{proof}
	Take $K(\om)>1$ such that $\{M_n\}_{n=0}^\infty$ is a lacunary sequence, and let $\{P_n\}_{n=0}^\infty$ be the sequence of polynomials associated to $N=1$ and $\{M_n\}_{n=0}^\infty$ via Theorem~\ref{th: descomposicion pavlovic}. 
	Observe that 
	$\Apq{1}{q} = \ell_{\frac{s}{q}}^q(H^1,\{P_n\}) $,  where $s=\log_2(K)$,  
	with equivalent norms. Now by Proposition~\ref{pr:dual-admisible}, Theorem~\ref{th:pavlovic-bloques} and the fact that $(H^1)^\ast\simeq [H^1,H^\infty] = BMOA$ via $H^2$-pairing \cite[Theorem~7.1 ]{girela} we obtain that
	\begin{equation}\label{eq:c21}
		(\Apq{1}{q})^\ast \simeq [\Apq{1}{q}, H^\infty] = [\ell^q_{\frac{s}{q}}(H^1,\{P_n\}), H^\infty] = \ell^{q'}_{-\frac{s}{q}}([H^1,H^\infty],\{P_n\}) = \ell^{q'}_{-\frac{s}{q}}(BMOA,\{P_n\})
	\end{equation}
	via the $H^2$-pairing with equivalence of norms. Now by Proposition~\ref{prop: estimacion norma poli} (with $\alpha=1$)
	$$
	\nm{I^\om P_n\ast g}_{BMOA} \asymp K^{-n}\nm{P_n \ast g}_{BMOA}, \quad n\in \N\cup \{0\},g\in \H(\D).
	$$
	
	Therefore, if $1<q<\infty$, by   Theorem~\ref{th: descomposicion pavlovic} and the previous estimate
	\begin{equation*}
		\begin{split}
			\nm{g}_{BMOA(q',\om)} &= \nm{g}_{\ell^{q'}_{\frac{s}{q'}}(BMOA,\{P_n\})} \\ &  \asymp  \sum_{n=0}^\infty K^{-n} \nm{P_n\ast g}_{BMOA}^{q'} \asymp \sum_{n=0}^\infty K^{-n+nq'} \nm{I^\om P_n\ast g}_{BMOA}^{q'} \\
			&= \sum_{n=0}^\infty K^{\frac{nq'}{q}}\nm{ P_n\ast I^\om g}^{q'}_{BMOA} 
			= \nm{I^\om g}_{\ell^{q'}_{\frac{-s}{q}}(BMOA,\{P_n\})}, \quad g\in \H(\D).
		\end{split}
	\end{equation*}
	Hence $I^\om:BMOA(q',\om) \to \ell^{q'}_{-\frac{s}{q}}(BMOA,\{P_n\})$ is isomorphic and onto. This together with \eqref{eq:c21} says that for any $q\in (1,\infty)$, $(\Apq{1}{q})^\ast$ is equivalent to $BMOA(q',\om)$ under the $A^2_\om$-pairing, with equivalence of norms.
	This finishes the proof of part (i).
	
	\par Assume that $0<q\le 1$. 
	By Proposition~\ref{prop: estimacion norma poli},
	$$
	\nm{I^{\mu_q} P_n\ast g}_{BMOA} \asymp K^{-(1+\frac{1}{q})n}\nm{P_n \ast g}_{BMOA}, \quad n\in \N\cup \{0\},\quad  g\in \H(\D).
	$$
	Then,  by   Theorem~\ref{th: descomposicion pavlovic} and the previous estimate
	\begin{equation*}
		\begin{split}
			\nm{g}_{BMOA(\infty,\om)} & \asymp  \nm{g}_{\ell^{\infty}_{s}(BMOA,\{P_n\})}
			\\
			&\asymp  \sup_{n\in \N\cup\{0\}} K^{-n} \nm{P_n\ast g}_{BMOA} \asymp \sup_{n\in \N\cup\{0\}} K^{-n+n(1+\frac{1}{q})} \nm{I^{\mu_q}( P_n\ast g)}_{BMOA}\\
			&=  \sup_{n\in \N\cup\{0\}} K^{\frac{n}{q}}\nm{P_n\ast  I^{\mu_q} g}_{BMOA} 
			= \nm{I^{\mu_q} g}_{\ell^{\infty}_{\frac{-s}{q}}(BMOA,\{P_n\})},\quad g\in \H(\D),
		\end{split}
	\end{equation*}
	and hence $I^{\mu_q}:BMOA(\infty,\om)\to \ell^{\infty}_{-\frac{s}{q}}(BMOA, \{P_n\})$ is isomorphic and onto. This together with \eqref{eq:c21} says that for any $0<q\leq 1$,  $(\Apq{1}{q})^\ast$ is equivalent to $BMOA(\infty,\om)$ under the $A^2_{\mu_q}$-pairing
	with equivalence of norms. This finishes the proof.
\end{proof}

\subsection{Proof of Theorem~\ref{th: dualidad general}. Cases (iii.a) and (iii.b).}
In order to deal with the range $0<p<1$ in the proof of  Theorem~\ref{th: dualidad general}, we will  use the following characterization of $(H^p)^\ast$.
For $0<p<1$ let us consider the standard weight  $\nu_{\frac{1}{p}-2}(z) =\left( \frac{1}{p}-1\right) (1-\abs{z}^2)^{\frac{1}{p}-2}$ and  the space

$$
I^{\nu_{\frac{1}{p}-2}}(\B)=\left\{ h\in \H(\D): h=I^{\nu_{\frac{1}{p}-2}}(g), \, g\in \B \right\},
$$
equipped with the norm $\nm{h}_{I^{\nu_{\frac{1}{p}-2}}(\B)} = \nm{g}_\B$.

We will  use the following characterization of $(H^p)^\ast$, $0<p<1$.

\begin{lettertheorem}\label{th: dual zhu}
	Let $0<p<1$ and $\nu_{\frac{1}{p}-2}(z) =\left( \frac{1}{p}-1\right) (1-\abs{z}^2)^{\frac{1}{p}-2}$.  Then $(H^p)^\ast \simeq I^{\nu_{\frac{1}{p}-2}}(\B) $ under the $H^2$-pairing with equivalence of norms .
\end{lettertheorem}
\begin{proof}
	By \cite[Theorem B]{zhu}, $(H^p)^\ast \simeq \B$ with equivalence of norms under the pairing
	$$
	\langle f , g\rangle_{p} = \lim_{r\to 1^-}\sum_{n=0}^\infty \widehat{f}(n)\overline{\widehat{g}(n)}\left (\nu_{\frac{1}{p}-2} \right )_{2n+1} r^n .$$
	Therefore, the result follows.
\end{proof}

\begin{lemma}\label{lema: bloch admisible}
	Let $0<p<1$, $0<q\leq \infty$ and  $\om\in \DD$. Then, 
	the following  statements hold:
	\begin{itemize}
		\item[(i)] $I^{\nu_{\frac{1}{p}-2}} :\B \to I^{\nu_{\frac{1}{p}-2}}(\B)$ is an isometric isomorphism;
		\item[(ii)] $I^{\nu_{\frac{1}{p}-2}}(\B)$ is an admissible Banach space;
		\item[(iii)] $I^{\nu_{\frac{1}{p}-2}} :\B(q,\om)\to I^{\nu_{\frac{1}{p}-2}} (q,\om)$ is an isometric isomorphism.
	\end{itemize}
\end{lemma}
\begin{proof}
	In order to simplify the notation, throughout the proof we write $X_p=I^{\nu_{\frac{1}{p}-2}}(\B)$ and $T_p=I^{\nu_{\frac{1}{p}-2}}$. Since (i) is clear, let us prove (ii).
	Firstly,  if $f\in X_p$ and $g=T_p^{-1}(f)\in \B$,  then $\sup_{n\in \N \cup \{0\}} \abs{\widehat{g}(n)} \lesssim \nm{g}_\B$. So if
	$0<R<1$,
	then
	$$
	\abs{f(z)} \leq \sum_{n=0}^{\infty}\abs{\widehat{g}(n)}\left (\nu_{\frac{1}{p}-2} \right)_{2n+1} \abs{z}^n \lesssim \nm{g}_\B \sum_{n=0}^{\infty}\left (\nu_{\frac{1}{p}-2} \right)_{2n+1} R^n = C(p,R) \nm{f}_{X_p},\quad |z|\le R,
	$$
	so the norm convergence in $X_p$  implies uniform convergence in compact subsets of $\D$. \\
	Now if $r>1$, 
	we will see that that 
	$$\nm{f}_{X_p} \lesssim \sup_{z\in\overline{D(0,R)}}\abs{f(z)}, \quad R=\frac{1+r}{2} \quad . $$ 
	
	If $f\in \H(D(0,r))$, then $g=T_p^{-1}(f)\in \H(D(0,r))$
	\begin{equation}\label{eq: norma xp 1}		
			\nm{f}_{X_p} = \nm{g}_\B \lesssim \sup_{z\in \overline{\D}} \abs{g(z)}.
	\end{equation}
	Now, let $R= \frac{1+r}{2}$ then by \cite[Proposition 6]{MorPeldelaRosa23} if $B_z$ are the reproducing kernels of $A^2$ and $a\in \overline{\D}$
	\begin{equation}\label{eq: norma xp 2}
	\abs{g(a)} =\abs{g_R \left(\frac{a}{R} \right)} \leq \int_{\D} \abs{f_R(\z)}\abs{T_p^{-1} \left(B_{\frac{a}{R}} \right)(\z)}dA(\z) \leq \sup_{z\in \overline{D(0,R)}}\abs{f(z)} \nm{T_p^{-1} \left(B_{\frac{a}{R}} \right)}_{A^1}.
	\end{equation}
	Now,  
	\begin{equation*}
		\begin{split}
			\nm{T_p^{-1} \left(B_{\frac{a}{R}} \right)}_{A^1} & \le \nm{T_p^{-1} \left(B_{\frac{a}{R}} \right)}_{A^2}
			\\ & \lesssim \left(  \sum_{n=0}^\infty \frac{(n+1)^2}{R^{2n} \left(\nu_{\frac{1}{p}-2}\right)_{2n+1}^2(n+1)} \right)^{1/2}
			\\ & \asymp \left(  \sum_{n=0}^\infty \frac{ (n+1)^{\frac{2}{p}-1} }{R^{2n} } \right)^{1/2}
			\\ & \asymp \frac{1}{(1-\frac{1}{R})^{\frac{1}{p}}} = C(p,r), 
		\end{split}
	\end{equation*}
	that combined with \eqref{eq: norma xp 1} and \eqref{eq: norma xp 2} gives,
	$$
	\nm{f}_{X_p} \lesssim \sup_{z\in \overline{D(0,R)}} \abs{f(z)},
	$$
	so the convergence in compact subsets of $D(0,r)$ implies the norm convergence in $X_p$. \\
	To conclude, if $\z \in \D$ by noting that $T_p(f_\z) = (T_p(f))_\z$ we obtain that if $f\in X_p$ and $g\in \B$ with $T_p(g)=f$, then 
	$$
	\nm{f_\z}_{X_p} = \nm{T_p(g_\z)}_{X_p} = \nm{g_\z}_\B \leq \nm{g}_\B = \nm{T_p(g)}_{X_p} = \nm{f}_{X_p}, 
	$$
	obtaining that $X_p$ is admissible. Moreover, since the classical Bloch space is Banach, then $X_p$ also is Banach. 
	Let us prove (iii). From (ii) and the properties of $T_p$ we obtain that for every $f\in \H(\D)$ and $0\leq r <1$, if $0<q<\infty$
	$$
	\nm{T_p f}_{X_p(q,\om)}^{q'} = \int_0^1 \nm{(T_p f)_r}_{X_p}^{q} \om(r)dr = \int_0^1 \nm{f_r}_{\B}^{q}\om(r)dr = \nm{f}_{\B(q,\om)}^{q},\quad f\in \H(\D),
	$$
	that is $T_p$ is an isometric isomorphism between  $X_p(q,\om)$ and $\B(q,\om)$, the case when $q=\infty$ is analogous to the previous case. This finishes the proof.
\end{proof}
Now we are ready we prove Theorem~\ref{th: dualidad general} (iii.a) and (iii.b).
\begin{proposition}\label{prop: dualidad p<1}
	Let $\om\in \DD$ and $0<p<1$.  Then,
	\begin{enumerate}
		\item[(i)] If $1<q<\infty$, then $(\Apq{p}{q})^\ast\simeq \B(q',\om)$ under the pairing
		$$
		\langle f , g \rangle_{\om,p} = \lim_{r\to 1^-} \sum_{n=0}^{\infty}\widehat{f}(n)\overline{\widehat{g}(n)} \om_{2n+1}\left (\nu_{\frac{1}{p}-2} \right )_{2n+1} r^n,\quad f,g\in \H(\D),
		$$
		with equivalence of norms; 
		\item[(ii)] If $0<q\leq 1$, then $(\Apq{p}{q})^\ast\simeq \B(\infty,\om)$ under the pairing
		$$
		\langle f , g \rangle_{\om,p,q} = \lim_{r\to 1^-} \sum_{n=0}^{\infty}\widehat{f}(n)\overline{\widehat{g}(n)} (\mu_{q})_{2n+1}\left (\nu_{\frac{1}{p}-2} \right )_{2n+1} r^n, \quad f,g\in \H(\D)
		$$
		with equivalence of norms, where $\mu_q = \om \omg^{\frac{1}{q}}$.
	\end{enumerate}
\end{proposition}
\begin{proof}
	In order to simplify the notation, throughout the proof we write $X_p=I^{\nu_{\frac{1}{p}-2}}(\B)$. 
	Let $N\in \N$ such that $p>\frac{1}{N+1}$ and $K(\om)>1$ such that $\{M_n\}_{n=0}^\infty$ is a lacunary sequence. Now,  let us consider $\{P_n\}_{n=0}^\infty$ be the sequence of polynomials associated to $N$ and $\{M_n \}_{n=0}^\infty$ via Theorem~\ref{th: descomposicion pavlovic}.
	If $s=\log_2(K)$,  
	by Theorem~\ref{th: descomposicion pavlovic} 
	\begin{equation*}\label{eq: prop norma Xpq}
		\Apq{p}{q} = \ell^{q}_\frac{s}{q}(H^p,\{P_n\})  
	\end{equation*}
	with equivalent norms. Therefore,  by Proposition~\ref{pr:dual-admisible}, Theorem~\ref{th:pavlovic-bloques} and Theorem~\ref{th: dual zhu}, we obtain that
	\begin{equation}\label{eq:c31}
		(\Apq{p}{q})^\ast \simeq [\Apq{p}{q},H^\infty] = [\ell^{q}_\frac{s}{q}(H^p,\{P_n\}), H^\infty] = \ell^{q'}_{-\frac{s}{q}}([H^p,H^\infty],\{P_n\}) = \ell^{q'}_{-\frac{s}{q}}(X_p,\{P_n\})
	\end{equation}
	via the $H^2$-pairing with equivalence of norms. Now, bearing in mind that  $X_p$ is admissible by Lemma~\ref{lema: bloch admisible}, and using Proposition~\ref{prop: estimacion norma poli} it follows that
	$$
	\nm{I^\om P_n\ast g}_{X_p} \asymp K^{-n}\nm{P_n \ast g}_{X_p}, \quad n\in \N\cup \{0\},g\in \H(\D).
	$$
	Then, if $1<q<\infty$,  by Theorem~\ref{th: descomposicion pavlovic} and the previous estimate
	\begin{equation*}
		\begin{split}
			\nm{g}_{X_p(q',\om)} &\asymp \nm{g}_{\ell^{q'}_{\frac{s}{q'}}(X_p,\{P_n\})}\asymp   \sum_{n=0}^\infty K^{-n+nq'} \nm{I^\om P_n\ast g}_{X_p}^{q'} \\
			&= \sum_{n=0}^\infty K^{n\frac{q'}{q}}\nm{  P_n\ast I^\om g}^{q'}_{X_p} = \nm{I^\om g}_{\ell^{q'}_{-\frac{s}{q}}(X_p,\{P_n\})}, \quad g\in \H(\D)
		\end{split}
	\end{equation*}
	and hence $I^\om:X_p(q',\om) \to \ell^{q'}_{-\frac{s}{q}}(X_p,\{P_n\})$ is isomorphic and onto. This together with \eqref{eq:c31} yields that   $(\Apq{p}{q})^\ast \simeq X_p(q',\om)$ via $A^2_\om$-pairing. So for every $L\in (\Apq{p}{q})^\ast$ there exists a $g\in X_p(q',\om)$ such that $L(f) = \langle f,g \rangle_{A^2_\om}$ for every $f\in \Apq{p}{q}$ and $\nm{L}_{(\Apq{p}{q})^\star} \asymp \nm{g}_{X_p}$, let now $h = T_p^{-1}(g)\in\B$, then by Lemma~\ref{lema: bloch admisible} (iv)
	$$
	L(f)=\langle f,g\rangle_{A_{\om}^2} = \langle f,T_p(h)\rangle_{A_{\om}^2} = \lim_{r\to 1^-} \sum_{n=0}^{\infty}\widehat{f}(n)\overline{\widehat{h}(n)} \om_{2n+1}\left (\nu_{\frac{1}{p}-2} \right )_{2n+1} r^n, \quad  f\in \Apq{p}{q}.
	$$ 
	Consequently,  $(\Apq{p}{q})^\ast \simeq \B(q',\om)$ under the pairing 
	$$
	\langle f , g \rangle_{\om,p} = \lim_{r\to 1^-} \sum_{n=0}^{\infty}\widehat{f}(n)\overline{\widehat{g}(n)} \om_{2n+1}\left (\nu_{\frac{1}{p}-2} \right )_{2n+1} r^n,\quad f,g\in \H(\D).
	$$
	with equivalence of norms, due to the equality $\nm{g}_{X_p}=\nm{h}_{\B}$. This  concludes the proof of part (i).
	
	In order to prove the part (ii), we observe that Lemma~\ref{lema: bloch admisible} and  Theorem~\ref{th: descomposicion pavlovic}  imply that  
	\begin{equation*}\label{eq: prop norma Xpinf}
		X_p(\infty,\om) = \ell^{\infty}_{s}(X_p,\{P_n\})
	\end{equation*}
	Now by Proposition~\ref{prop: estimacion norma poli}
	$$
	\nm{I^{\mu_q} P_n\ast g}_{X_p} \asymp K^{-(1+\frac{1}{q})n}\nm{P_n\ast g}_{X_p},\quad n\in \N\cup \{0\},\quad g\in \H(\D).
	$$
	Therefore, 
	\begin{equation*}
		\begin{split}
			\nm{g}_{X_p(\infty,\om)} &\asymp  \sup_{n\in \N\cup\{0\}} K^{-n} \nm{P_n\ast g}_{X_p} \asymp \sup_{n\in \N\cup\{0\}} K^{-n+n(1+\frac{1}{q})} \nm{I^{\mu_q} P_n\ast g}_{X_p} \\
			&=  \sup_{n\in \N\cup\{0\}} K^{\frac{n}{q}}\nm{P_n\ast I^{\mu_q} g}_{X_p} 
			= \nm{I^{\mu_q} g}_{\ell^{\infty}_{\frac{-s}{q}}(X_p,\{P_n\})},\quad g\in \H(\D),
		\end{split}
	\end{equation*}
	and hence $I^{\mu_q}: X_p(\infty,\om)\to \ell^{\infty}_{-\frac{s}{q}}(X_p,\{P_n\})$ is isomorphic and onto.  Joinig this assertion with \eqref{eq:c31} we obtain that $(\Apq{p}{q})^\ast \simeq X_p(\infty,\om)$ via the 
	$A^2_{\mu_q}$-pairing. By applying again Lemma~\ref{lema: bloch admisible} we obtain that $(\Apq{p}{q})^\ast \simeq \B(\infty,\om)$ under the pairing
	$$
	\langle f , g \rangle_{\om,p,q} = \lim_{r\to 1^-} \sum_{n=0}^{\infty}\widehat{f}(n)\overline{\widehat{g}(n)} \mu_{2n+1}\left (\nu_{\frac{1}{p}-2} \right )_{2n+1} r^n, \quad f,g\in \H(\D),
	$$
	with equivalence of norms, concluding the proof.
\end{proof}
\section{Boundedness of $P_\om$ and $P_\om^+$ in $L^{p,q}_\om$}
We begin this section proving that $P_\om$ is bounded on $L^{p,q}_\om$ whenever $1<p,q<\infty$ and $\om\in\DD$.

\begin{Prf} {\em{Theorem~\ref{coro: acotacion proyeccion}}.}
	For each $g\in L^{p,q}_\om$ let us   consider the functional
	$$
	L_g(f) = \langle f, g\rangle_{A^2_\om}, \quad f\in \Apq{p'}{q'}.
	$$
	By H\"older's inequality 
	$$
	\abs{L_g(f)} \le \int_0^1 M_p(r,g)M_{p'}(r,f)\om(r) dr \leq \nm{g}_{L^{p,q}_\om}\nm{f}_{\Apq{p'}{q'}},
	$$
	therefore $L_g\in (\Apq{p'}{q'})^\ast$ and $\nm{L_g}_{(\Apq{p'}{q'})^\ast}\le \nm{g}_{ L^{p,q}_\om}$. In addition,  since $L^{p,q}_\om \subset L^1_\om$, by \cite[2.3]{advances} 
	$$
	L_g(f) = \lim_{r\to 1^-} \langle f_r,g \rangle_{L^2_\om} = \lim_{r\to 1^-} \langle f_r,P_\om(g) \rangle_{L^2_\om}, \quad  f\in \Apq{p'}{q'}.
	$$
	On the other hand, by Theorem~\ref{th: dualidad general} (i.a) there exists $h\in \Apq{p}{q}$ such that 
	$$
	L_g(f) = \langle f, h\rangle_{A^2_\om}, \quad f\in \Apq{p'}{q'},
	$$ 
	and $\nm{L_g}_{(\Apq{p'}{q'})^\ast} \asymp \nm{h}_{\Apq{p}{q}}$.
	Then, for each $n\in\N\cup\{0\}$, $ \langle z^n, h\rangle_{A^2_\om}= \langle  z^n, P_\om(g) \rangle_{A^2_\om}$, which implies that
	$h=P_\om(g)$. Moreover,
	$$
	\nm{P_\om(g)}_{L^{p,q}_\om} = \nm{h}_{\Apq{p}{q}} \asymp \nm{L_g}_{(\Apq{p'}{q'})^\ast} \le \nm{g}_{L^{p,q}_\om},
	$$
	that is $P_\om$ is bounded on  $L^{p,q}_\om$ .
\end{Prf} 

\medskip

Now we will prove Proposition~\ref{pr:kernels-p=1}. With this aim  let us recall the following result.
\begin{letterproposition}\label{pr:cesaro}
	Let $k \in \N$, $k>1$ and $\Psi:\mathbb{R}\to\mathbb{R}$ be a $C^\infty$-function such that $\Psi\equiv 1$ on $(-\infty,1]$, $\Psi\equiv 0$ on $[k,\infty)$ and $\Psi$ is decreasing and positive on $(1,k)$. Set $\psi(t)=\Psi\left(\frac{t}{k}\right)-\Psi(t)$ for all $t\in\mathbb{R}$. Let $V_{0,k}(z)=\sum\limits_{j=0}^{k-1} \Psi(j) z^j$
	and
	\begin{equation*}\label{vnk}
		V_{n,k}(z) 
		=\sum_{j=0}^\infty
		\psi\left(\frac{j}{k^{n-1}}\right)z^j=\sum_{j=k^{n-1}}^{k^{n+1}-1}
		\psi\left(\frac{j}{k^{n-1}}\right)z^j,\quad n\in\N.
	\end{equation*}
	
	Then,
	\begin{equation}
		\label{propervn1}
		f(z)=\sum_{n=0}^\infty (V_{n,k}\ast f)(z),\quad z\in\D,\quad f\in\H(\D),
	\end{equation}
	and for each $0<p<\infty$ there exists a constant $C=C(p,\Psi,k)>0$ such that
	\begin{equation}
		\label{propervn2}
		\|V_{n,k}\ast f\|_{H^p}\le C\|f\|_{H^p},\quad f\in H^p, \quad n \in \N.
	\end{equation}
\end{letterproposition}

If $k=2$ we simply denote $V_{n,2}=V_n$. A
proof of Proposition~\ref{pr:cesaro} for this choice appears in
\cite[p. 143--144]{Pabook2}. A proof for any $k>1$ can be found in \cite[Proposition~4]{PeldelaRosa22}.

\begin{Prf}{\em{  Proposition~\ref{pr:kernels-p=1}}.}
	
	Since $M_{1}(r, (B^\om_a)^{(N)})=\| (B^\om_{|a|r})^{(N)}\|_{H^1}$, it is enough to prove that
	\begin{equation}\label{eq:s1}\| (B^\om_{|a|})^{(N)}\|_{H^1}\asymp 1+\int_{0}^{|a|}\frac{dt}{(1-t)^{N+1}\omg(t)},\quad a \in \D.
	\end{equation}
	
	Firstly, let us prove the inequality 
	$$\| (B^\om_{|a|})^{(N)}\|_{H^1}\lesssim 1+\int_{0}^{|a|}\frac{dt}{(1-t)^{N+1}\omg(t)},\quad a \in \D,$$
	for $N\ge 1$.
	Since   \begin{equation}\label{kernelformula}
		B^\om_a(z)=\sum_{n=0}^\infty\frac{(z\overline{a})^n}{2\om_n},
	\end{equation}
	we have that
	\begin{displaymath}
		\left(B^\om_a\right)^{(N)}(z)=\sum_{j=N}^\infty
		\frac{j(j-1)\cdots(j-N+1)z^{j-N}\overline{a}^{j}}{2\om_{2j+1}},\quad N\in\N.
	\end{displaymath}
	For any $g(z)=\sum_{n=0}^\infty \widehat{g}(n)z^n \in \H(\D)$, let us denote $D^\omega g(z)=\sum_{n=0}^\infty \frac{\widehat{g}(n)}{2\om_{2n+1}}z^n$.
	Therefore,  by Proposition~\ref{pr:cesaro} and  \cite[Lemma~3.1]{MatPavStu84} 
	\begin{equation}\label{eq:s2}
		\| (B^\om_{|a|})^{(N)}\|_{H^1}\le \sum_{n=0}^\infty \| V_n\ast (B^\om_{|a|})^{(N)}\|_{H^1}\lesssim 1+  \sum_{n=1}^\infty |a|^{2^{n-1}} \| D^\om  V_n*F^N\|_{H^1}, \quad a\in \D,
	\end{equation}
	where  $F^N(z)=\frac{d^{(N)}}{dz^N}\left( \frac{1}{1-z}\right)$.
	Now, we will prove that
	\begin{equation}\label{eq:s3}
		\|D^\om V_n*g\|_{H^1}\lesssim  \om_{2^{n+2}}^{-1}  \| V_n*g\|_{H^1}, \quad n \in \N, \, g\in \H(\D),
	\end{equation}
	where the constants involved do not depend on $n$ or on $g(z)=\sum_{j=0}^\infty \widehat{g}(j)z^j$. Let us denote $\widehat{G}(j)(z)=\sum_{k=0}^j \widehat{V_n}(j)\widehat{g}(j)z^j$. Hence a summation by parts and the fact that $\widehat{G}(j)=0$ for $j\le 2^{n-1}-1$ yields
	
	\begin{equation*}\begin{split}
			D^\om V_n*g=\sum_{j=2^{n-1}}^{2^{n+1}-2}\widehat{G}(j)\left( \frac{1}{\om_{2j+1}}- \frac{1}{\om_{2j+3}} \right)+  \frac{1}{\om_{2^{n+2}-1}}V_n*g.
	\end{split}\end{equation*}
	Next, bearing in mid that 
	$\widehat{G}(j) = (j+1)\sigma_j V_n*g-j\sigma_{j-1}V_n*g$, a new summation by parts gives 
	\begin{equation}\begin{split}\label{eq:s4}
			D^\om V_n*g &=\sum_{j=2^{n-1}}^{2^{n+1}-3}(j+1)\sigma_j V_n*g 
			\left[\left( \frac{1}{\om_{2j+1}}- \frac{1}{\om_{2j+3}} \right)- \left( \frac{1}{\om_{2j+3}}- \frac{1}{\om_{2j+5}}\right)\right]+ 
			\\ & + 2^{n+1-2}  \sigma_{2^{n+1}-2} V_n*g   \left( \frac{1}{ \om_{2^{n+2}-3} }- \frac{1}{ \om_{2^{n+2}-1} } \right)  +\frac{1}{\om_{2^{n+2}-1}}V_n*g.
	\end{split}\end{equation}

	Therefore, by Theorem~\ref{th:cesaro} (see also \cite{HL} for a proof for $X=H^p$, $\frac{1}{2}<p<\infty$) and Lemma~\ref{lemma: caract dgorro}(vi)
	\begin{equation*}\begin{split}
			& \left\|\sum_{j=2^{n-1}}^{2^{n+1}-3}(j+1)\sigma_j V_n*g 
			\left[\left( \frac{1}{\om_{2j+1}}- \frac{1}{\om_{2j+3}} \right)- \left( \frac{1}{\om_{2j+3}}- \frac{1}{\om_{2j+5}}\right)\right] \right\|_{H^1}
			\\
			&\le \| V_n*g\|_{H^1} \sum_{j=2^{n-1}}^{2^{n+1}-3}\left|(\om_{2j+1}^{-1}-\om_{2j+3}^{-1})-(\om_{2j+3}^{-1}-\om_{2j+5}^{-1})\right|(j+1)\\
			&=\| V_n*g\|_{H^1} \sum_{j=2^{n-1}}^{2^{n+1}-3} \left|\frac{\left(\om_{\{2\}}\right)_{2j+1}\om_{2j+5}-\left(\om_{\{1\}}\right)_{2j+3}\left(\left(\om_{\{1\}}\right)_{2j+1}+\left(\om_{\{1\}}\right)_{2j+3}\right)}
			{\om_{2j+1}\om_{2j+3}\om_{2j+5}}\right|(j+1)\\
			&\le \| V_n*g\|_{H^1}   \sum_{j=2^{n-1}}^{2^{n+1}-3}(j+1)
			\left(\frac{(\om_{\{2\}})_{2j+1}}{\om_{2j+1}\om_{2j+3}}+ 2\frac{(\om_{\{1\}})^2_{2j+1}}{\om_{2j+1}\om_{2j+3}\om_{2j+5}} \right)\\
			& \lesssim \| V_n*g\|_{H^1}   \sum_{j=2^{n-1}}^{2^{n+1}-3}\frac{1}{(j+1)\om_{2j+5}}
			\\ & \lesssim    \om_{2^{n+2}}^{-1}  \| V_n*g\|_{H^1}, n\in\N,
	\end{split}\end{equation*}
	and
	\begin{equation*}
		\begin{split}
			&\left\|   2^{n+1-2}  \sigma_{2^{n+1}-2} V_n*g   \left( \frac{1}{ \om_{2^{n+2}-3} }- \frac{1}{ \om_{2^{n+2}-1} } \right)  \right\|_{H^1}\\
			&\le \| V_n*g\|_{H^1} 2^{n+1-2}   \left( \frac{1}{ \om_{2^{n+2}-3} }- \frac{1}{ \om_{2^{n+2}-1} } \right) \\
			&\lesssim \om_{2^{n+2}}^{-1}  \| V_n*g\|_{H^1}, n\in\N,
		\end{split}
	\end{equation*}
	Consequently, joining the above two chain of inequalities together with \eqref{eq:s4}, we get \eqref{eq:s3}. Next, putting together \eqref{eq:s2},   \eqref{eq:s3},  \cite[Lemma 3.1]{MatPavStu84} and \eqref{propervn2} it follows that
	
	\begin{equation}\begin{split}\label{eq:s5}
			\| (B^\om_{|a|})^{(N)}\|_{H^1} &\lesssim  1+  \sum_{n=1}^\infty |a|^{2^{n-1}}\om_{2^{n+2}}^{-1}  \| V_n*F^N\|_{H^1},
			\\ & \asymp 1 +  \sum_{n=1}^\infty |a|^{2^{n-1}}\om_{2^{n+2}}^{-1} M_1 \left(1-\frac{1}{2^{n+1}},  V_n*F^N\right) 
			\\ & \lesssim 1 +  \sum_{n=1}^\infty |a|^{2^{n-1}}\om_{2^{n+2}}^{-1} M_1 \left(1-\frac{1}{2^{n+1}},  F^N\right) 
			\\ & \asymp 1 +  \sum_{n=1}^\infty 2^{nN}|a|^{2^{n-1}}\om_{2^{n+2}}^{-1},
			\quad a\in \D,
	\end{split}\end{equation}
	where in the last equivalence we use that $N\ge 1$.
	Now, by Lemma~\ref{lemma: caract dgorro}(v)
	\begin{equation}\begin{split}\label{eq:s6}
			\sum_{n=1}^\infty 2^{nN}|a|^{2^{n-1}}\om_{2^{n+2}}^{-1}
			& \asymp \sum_{n=1}^\infty 2^{n(N-1)}|a|^{2^{n-1}}\om_{2^{n+2}}^{-1} (2^n-2^{n-1})
			\\ & \lesssim \sum_{j=1}^\infty j^{N-1}|a|^j \om_{2j}^{-1}, 
			\quad a\in \D. 
	\end{split}\end{equation}
	
	It is clear that it is enough to prove \eqref{eq:s1} for $|a|\ge \frac{3}{4}$. Therefore, take $a \in \D$ with  $|a|\ge \frac{3}{4}$ and choose $N^\star\in \N$ such that $1-\frac{1}{N^\star}\le |a|<1-\frac{1}{N^\star+1}.$
	Then,  by Lemma~\ref{lemma: caract dgorro}(iii) and (iv)
	\begin{equation}\begin{split}\label{eq:s7}
			\sum_{j=1}^{N^\star} j^{N-1}|a |^j \om^{-1}_{2j} & \asymp  \sum_{j=1}^{N^\star} \frac{ j^{N-1}|a|^j }{\omg\left(1-j^{-1} \right)}
			\\ & \lesssim 1+\int_{1}^{N^\star} \frac{x^{N-1}}{\omg(1-x^{-1})}\,dx
			\\ & \asymp 1+\int_{0}^{1-\frac{1}{N^\star}}\frac{dt}{(1-t)^{N+1}\omg(t)}
			\\ & \le  1+\int_{0}^{|a|}\frac{dt}{(1-t)^{N+1}\omg(t)},\quad  |a|\ge \frac{3}{4}.
	\end{split}\end{equation}
	Now  by Lemma~\ref{lemma: caract dgorro}(ii) there exists $\b>0$ such that $(1-r)^{-\beta}\omg(r)$ is an essentially increasing function. Therefore
	
	\begin{equation}\begin{split}\label{eq:s8}
			\sum_{j=N^\star+1}^\infty j^{N-1}|a|^j \om^{-1}_{2j} & \asymp  \ \sum_{j=N^\star+1}^\infty \frac{ j^{N+\b-1}|a|^j }{j^\beta\omg\left(1-j^{-1} \right)}
			\\ & \lesssim \frac{1}{(N^\star+1)^\b \omg\left(1-(N^\star+1)^{-1} \right)} \sum_{j=N^\star+1}^\infty  j^{N+\b-1}|a|^j
			\\ & \lesssim \frac{1}{(N^\star+1)^\b \omg\left(1-(N^\star+1)^{-1} \right)} \frac{1}{(1-|a|)^{N+\beta}}
			\\ & \asymp \frac{1}{(1-|a|)^{N}\omg(|a|)}
			\\ & \lesssim \int_{2|a|-1}^{|a|}\frac{dt}{(1-t)^{N+1}\omg(t)}\le \int_{0}^{|a|}\frac{dt}{(1-t)^{N+1}\omg(t)},
			\quad  |a|\ge \frac{3}{4}.
	\end{split}\end{equation}
	Therefore joining \eqref{eq:s5},   \eqref{eq:s6},  \eqref{eq:s7} and  \eqref{eq:s8} we get the inequality
	\begin{equation}\label{eq:s9}
		\| (B^\om_{|a|})^{(N)}\|_{H^1}\lesssim 1+\int_{0}^{|a|}\frac{dt}{(1-t)^{N+1}\omg(t)},\quad a \in \D, N\ge 1.
	\end{equation}
	In particular,
	we have that 
	\begin{equation*}\begin{split}
			M_1(s, (B^\om_{|a|})') & \lesssim 1+\int_{0}^{s|a|}\frac{dt}{(1-t)^{2}\omg(t)}
			\\ & =  1+\int_{0}^{s}\frac{|a|dt}{(1-|a|t)^{2}\omg(|a|t)}\lesssim \frac{1}{(1-|a|s)\omg(|a|s)},\quad 0\le s<1,\quad a \in \D.
	\end{split}\end{equation*}
	Therefore
	\begin{equation*}\begin{split}
			M_1(r, (B^\om_{|a|}) )  & \lesssim 1+ \int_{0}^{r} M_1(s, (B^\om_{|a|})')\,ds \lesssim  1+ \int_{0}^{r}  \frac{ds}{(1-|a|s)\omg(|a|s)}
			\\ & \asymp  \frac{1}{|a|}\int_{0}^{r|a|}  \frac{ds}{(1-s)\omg(s)}\lesssim \int_{0}^{r|a|}  \frac{ds}{(1-s)\omg(s)}  ,\quad 0<r<1,\, \, |a| \ge 1/2,
	\end{split}\end{equation*}
	which implies \eqref{eq:s9} for $N=0$.

	\par As for the reverse ineqaulity in \eqref{eq:s1}, observe that 
	using the classical Hardy's inequality \cite[p. 48]{Duren} it follows that
	$$ \| (B^\om_{|a|})^{(N)}\|_{H^1} \gtrsim  1+ \sum_{j=N}^\infty \frac{|a|^j j^{N-1}}{\om_{2j+1}},$$
	therefore a similar argument to \eqref{eq:s7}-\eqref{eq:s8} gives that 
	$$ \| (B^\om_{|a|})^{(N)}\|_{H^1} \gtrsim 1+\int_{0}^{|a|}\frac{dt}{(1-t)^{N+1}\omg(t)}, \quad a \in \D, N\in \N\cup\{0\},$$
	see  \cite[Lemma~10]{PelRatproj} for further details. Consequently, \eqref{eq:s1} holds and 
	this finishes the proof.
	
\end{Prf}

Now we are ready to prove Theorem~\ref{th: acotacion P+}.

\begin{Prf}{\em{Theorem~\ref{th: acotacion P+}.}}
	Assume that $P_\om^+:L^{p,q}_\om \to L^{p,q}_\om$ is bounded. For every $0< t <1$ consider the test functions $f_t(z)=t^{-1/q}\chi_{\D\setminus D(0,t)}$ and denote $J_\om(t)=\int_0^t\frac{ds}{\omg(s)(1-s)}$ for all $0\leq t< 1$. Bearing in mind that  $\nm{f_t}_{L^{p,q}_\om}^q = \omg(t)$ and Proposition~\ref{pr:kernels-p=1}
	\begin{equation*}
		\begin{split}
			\omg(t)\nm{P_\om^+}_{L^{p,q}_\om \to L^{p,q}_\om}^q &\geq \int_0^1\left (\frac{1}{2\pi}\int_0^{2\pi}\left (\int_\D f_t(z)\abs{B^\om_{re^{i\t}}(z)}\om(z)dA(z) \right)^p d\t \right )^{\frac{q}{p}}\, r\om(r)dr \\
			&= \int_0^1\left (\frac{1}{2\pi}\int_0^{2\pi}\left (\int_0^1 f_t(s) M_1(s,B^\om_{re^{i\t}}) s\om(s)\,ds \right)^p d\t \right )^{\frac{q}{p}}\, r\om(r)dr \\
			&\asymp \int_0^1\left (\frac{1}{2\pi}\int_0^{2\pi}\left (\int_0^1 f_t(s) (J_\om(sr)+1) s\om(s)\,ds \right)^p d\t \right )^{\frac{q}{p}}\, r\om(r)dr \\
			&= \int_0^1\left (\int_0^1 f_t(s) (J_\om(sr)+1) s\om(s) ds \right )^q \, r\om(r)dr \\
			&\geq \int_0^1\left (\int_r^1 f_t(s) (J_\om(sr)+1) s\om(s) ds \right )^q \, r\om(r)dr \\
			&\geq \int_0^1\left (\int_r^1 f_t(s) s\om(s)ds \right )^q (J_\om(r^2)+1)^q\, r\om(r)dr \\
			&\geq \int_0^t \left (\int_r^1 f_t(s) s\om(s)ds \right )^q (J_\om(r^2)+1)^q\, r\om(r)dr \\
			&\gtrsim \left (\int_{1/2}^t (J_\om(r^2)+1)^q \om(r)dr \right ) \omg(t)^q \\
			&\asymp \left (\int_{1/2}^t (J_\om(r)+1)^q \om(r)dr \right ) \omg(t)^q,
		\end{split}
	\end{equation*}
	where
	the last step follows from Lemma~\ref{lemma: caract dgorro} (ii). Therefore, 
	$$
	\sup_{1/2\le t<1} \left (\int_{1/2}^t (J_\om(r)+1)^q\om(r) \right )^{\frac{1}{q}}\omg(t)^{\frac{1}{q'}} < \infty.
	$$
	Consequently, arguing as in the
	the proof of \cite[Theorem 9]{advances},    $\om\in \Dd$. \\
	Reciprocally assume that $\om\in\Dd$. By
	Minkowski's inequality,
	\begin{equation*}
		\begin{split}
			\frac{1}{2}M_p(r,P_\om^+f) &= \left (\frac{1}{2\pi}\int_0^{2\pi}\left | \int_0^1\left ( \frac{1}{2\pi}\int_0^{2\pi}f(se^{i\t})\abs{B^\om_{se^{i\t}}(re^{it})}d\t\right )s\om(s)ds \right |^p dt \right  )^{\frac{1}{p}} \\
			&\leq \int_0^1 \left ( \frac{1}{2\pi}\int_0^{2\pi} \left (\frac{1}{2\pi}\int_0^{2\pi}\abs{f(se^{i\t})}\abs{B^\om_{se^{i\t}}(re^{it})} d\t\right )^p dt\right )^{\frac{1}{p}} s\om(s) ds.
		\end{split}
	\end{equation*}
	Next, by  H\"older's inequality and Fubini's theorem
	\begin{equation*}
		\begin{split}
			&\frac{1}{2\pi}\int_0^{2\pi} \left (\frac{1}{2\pi}\int_0^{2\pi}\abs{f(se^{i\t})}\abs{B^\om_{se^{i\t}}(re^{it})} d\t \right )^p dt \\ 
			\leq &\frac{1}{2\pi}\int_0^{2\pi}\left ( \frac{1}{2\pi}\int_0^{2\pi}\abs{f(se^{i\t})}^p\abs{B^\om_{se^{i\t}}(re^{it})}d\t\right )\left ( \frac{1}{2\pi} \int_0^{2\pi} \abs{B^\om_{se^{i\t}}(re^{it})} d\t\right )^{\frac{p}{p'}}dt \\
			= &M_1^{\frac{p}{p'}}(s,B^\om_r) \frac{1}{2\pi}\int_0^{2\pi}\left ( \frac{1}{2\pi}\int_0^{2\pi}\abs{f(se^{i\t})}^p\abs{B^\om_{se^{i\t}}(re^{it})}d\t\right ) dt 
			=  M_1^{p}(s,B^\om_r)M_p^p(r,f).
		\end{split}
	\end{equation*}
	Therefore,
	$$
	M_p(r,P_\om^+f) \leq \int_0^1 M_p(s,f)M_1(s,B^\om_r) s\om(s) ds.
	$$
	Next consider the auxiliary radial function $h(r) = \omg(r)^{-\frac{1}{qq'}}$. Then,  by  H\"older's inequality
	\begin{equation*}
		\begin{split}
			\nm{P_\om^+f}_{L^{p,q}_\om}^q &\leq \int_0^1 \left ( \int_0^1 M_p(s,f)h(s)^{-1}h(s)M_1(s,B^\om_r) s\om(s)ds \right )^q r \om(r)dr \\
			&\leq \int_0^1 \left ( \int_0^1 M_p^q(s,f)h(s)^{-q} M_1(B^\om_r,s) s \om(s) ds \right )\left ( \int_0^1 h(s)^{q'}M_1(B^\om_r,s) s\om(s)ds\right )^{\frac{q}{q'}} r\om(r)dr.
		\end{split}
	\end{equation*}
	Using Proposition~\ref{pr:kernels-p=1} and Lemma~\ref{caract. D check} (ii) 
	\begin{equation*}\label{eq: estimacion test schur 1}
		\begin{split}
			\int_0^1 h(s)^{q'}M_1(B^\om_r,s) s\om(s)ds &\lesssim \int_0^1 \frac{\om(s)}{\omg(s)^{\frac{1}{q}} }\left ( 1+\int_0^{rs}\frac{dt}{\omg(t)(1-t)}dt\right )  ds \\
			&\asymp \int_0^1 \frac{\om(s)}{\omg(s)^{\frac{1}{q}} \omg(rs)}ds \\
			&\leq \int_0^r \frac{\om(s)}{\omg(s)^{1+\frac{1}{q}}}ds + \frac{1}{\omg(r)}\int_r^1 \frac{\om(s)}{\omg(s)^{\frac{1}{q}}}ds  \\ 
			&\lesssim \frac{1}{\omg(r)^\frac{1}{q}} = h(r)^{q'},\quad 0\leq r < 1.
		\end{split}
	\end{equation*}
	And analogously it can be proved that
	$$
	\int_0^1 h(r)^{q}M_1(B^\om_s,r) r\om(r)dr \lesssim h(s)^q, \quad 0\leq s < 1.
	$$
	Using both inequalities  and Fubini's Theorem
	\begin{equation*}
		\begin{split}
			\nm{P_\om^+f}_{L^{p,q}_\om}^q &\lesssim \int_0^1 h(r)^{q} \left ( \int_0^1M_p^q(s,f)h(s)^{-q}M_1(B^\om_r,s) s \om(s)ds \right ) r\om(r) dr \\
			&= \int_0^1 M_p^q(s,f)h(s)^{-q} \left ( \int_0^1 h(r)^{q}M_1(B^\om_s,r)\om(r)dr \right ) s \om(s) ds \\
			&\lesssim \int_0^1 M_p^q(s,f) s\om(s) ds = \nm{f}_{L^{p,q}_\om}^q.
		\end{split}	
	\end{equation*}
	This finishes the proof.
\end{Prf}

\end{document}